\documentclass[12pt]{article}
\usepackage{pdfrender,xcolor}
\usepackage{graphicx,color}
\usepackage{amsmath,amsfonts,amssymb,amsthm}
\usepackage[hidelinks,colorlinks=true,linkcolor=blue,citecolor=blue]{hyperref}
\usepackage{rsfso}
\usepackage{comment}
\usepackage{enumerate}

\usepackage[numbers]{natbib}
\newcommand{\ds}{\displaystyle }

\newcommand{\vc}[1]{{\boldsymbol #1}} 
\newcommand{\vcn}[1]{{\bf #1}}

\newcommand{\sr}[1]{{\mathcal #1}}
\newcommand{\dd}[1]{\mathbb{#1}}

\newcommand{\eq}[1]{(\ref{eq:#1})}
\newcommand{\lem}[1]{Lemma~\ref{lem:#1}}
\newcommand{\cor}[1]{Corollary~\ref{cor:#1}}
\newcommand{\thr}[1]{Theorem~\ref{thr:#1}}

\newcommand{\rem}[1]{Remark~\ref{rem:#1}}

\newcommand{\app}[1]{Appendix~\ref{app:#1}}
\newcommand{\sectn}[1]{Section~\ref{sec:#1}}

\newcommand{\lemt}[1]{\ref{lem:#1}}

\newcommand{\sect}[1]{\ref{sec:#1}}

\newcommand{\pend}{\hfill \thicklines \framebox(6.6,6.6)[l]{}}

\newenvironment{proof*}[1]{\noindent {\sc  #1} \rm}{\pend}
\newenvironment{keyword}[1]{ {\bf #1} \rm}{}

\newtheorem{theorem}{Theorem}[section]
\newtheorem{lemma}{Lemma}[section]

\newtheorem{remark}{Remark}[section]

\newtheorem{corollary}{Corollary}[section]

\newenvironment{mylist}[1]{\begin{list}{}
{\setlength{\itemindent}{#1mm}}
{\setlength{\itemsep}{0ex plus 0.2ex}}
{\setlength{\parsep}{0.5ex plus 0.2ex}}
{\setlength{\labelwidth}{10mm}}
}{\end{list}}

\newcommand{\setnewcounter} {
\setcounter{subsection}{0}
\setcounter{equation}{0}
\setcounter{figure}{0}
\setcounter{conjecture}{0}
\setcounter{assumption}{0}
\setcounter{question}{0}
\setcounter{definition}{0}
\setcounter{theorem}{0}
\setcounter{corollary}{0}
\setcounter{lemma}{0}
\setcounter{proposition}{0}
\setcounter{remark}{0}
}

\begin{document}
 \title{\Large \bf Heavy traffic limit of stationary distribution of\\ the multi-level single server queue}

\author{Masahiro Kobayashi\\ Tokai University \and Masakiyo Miyazawa\\ Tokyo University of Science \and Yutaka Sakuma \\National Defense Academy of Japan}
\date{\today}

\maketitle

\begin{abstract}

\citet{AtarMiya2025} recently introduced a single server queue with queue length dependent arrival and service processes, and name it a multi-level queue. They prove that the heavy traffic limit of its queue length process weakly converges to a reflected diffusion with discontinuously state-dependent drift and deviations.

We derive the heavy traffic limit of the stationary queue length distribution of this multi-level queue in a closed form, which agrees with the stationary distribution of the reflected diffusion obtained by \citet{Miya2024b}. Thus, those results show the limit interchange of process and stationary distribution in heavy traffic.

The multi-level queue of \cite{AtarMiya2025} is a simpler version of the 2-level $GI/G/1$ queue of \citet{Miya2025} and its extension for multi-levels. For this 2-level queue in heavy traffic, the process limit is unknown, and the distributional limit is obtained for limited cases under extra conditions. Nevertheless, it is shown that the method developed in \cite{Miya2025} perfectly works for the present multi-level queue.

\end{abstract}

\keyword{Keywords:}{single server queue, level dependent arrival and service, stationary distribution, heavy traffic, diffusion approximation, Palm distribution.}


\section{Introduction}
\label{sec:introduction}
\setnewcounter

Recently, \citet{AtarMiya2025} introduced a multi-level queue, and prove that the heavy traffic limit of its queue length process is a reflected diffusion with discontinuously state dependent drift and deviations. This multi-level queue is a simpler version of the 2-level $GI/G/1$ queue studied in \cite{Miya2025}, and has a single server and finitely many levels which constitute a partition of $\dd{R}_{+} \equiv [0,\infty)$ of intervals, and service speed depends on the level to which the queue length belongs. These assumptions are exactly the same as the 2-level $GI/G/1$ queue. However, their arrival processes are different. In the multi-level queue of \cite{AtarMiya2025}, the arrival times of customers are generated by a single renewal process whose time scale is changed similarly to the service speed, while different renewal processes are chosen similarly depending on the queue length in the 2-level $GI/G/1$ queue.

We consider the multi-level queue of \cite{AtarMiya2025}, but focus on the heavy traffic limit of the stationary distribution of its queue length, which is not addressed in \cite{AtarMiya2025} but important in application. Our interest is in the existence of this distributional limit and its closed form expression. This existence follows from the tightness of the sequence of the stationary distributions of the diffusion scaled queue length processes and \lem{th1} in \sectn{concluding} (see \cite{Bill2000} for the definition of the tightness). Furthermore, also by \lem{th1}, if the distributional limit exists, it is identical with the stationary distribution of the reflected diffusion of \cite{AtarMiya2025} by the result of \cite{Miya2024b} (see \rem{MLQ-main 2} for details).

Thus, the tightness is a key to prove the existence of the distributional limit of our interest. There are two approaches to prove this tightness in general. One is to use the corresponding process limit and a similar result to \lem{th1} when they are available. This line of this approach is well developed in the literature (e.g., see \cite{BudhLee2009}), and used to show the so called limit interchange, meaning that process and distributional limits are interchangeable. However, this approach may be hard to apply to queues with state-dependent arrivals and services as will be discussed in \sectn{concluding}. Another is to use the stationary equation or directly compute the stationary distribution. This approach is applicable when the process limit is not available, but needs to show the uniqueness of the distributional limit in addition to the tightness. Its typical approach is to consider the sequence of the stationary equations, each of which is called a basic adjoint relationship, BAR for short. This BAR approach has been successfully employed in queues and their networks with arrivals and services not depending on queue lengths (e.g., see \cite{BravDaiMiya2017,BravDaiMiya2024,Miya2017}).

In this paper, we employ the BAR approach, particularly following it developed for the 2-level $GI/G/1$ queue in \cite{Miya2025}. For this  $GI/G/1$ queue, the process limit is not available, and the heavy traffic limit of the stationary distributions is obtained in limited cases under some extra conditions (see Theorem 3.1 of \cite{Miya2025}). Contrary to this situation, the process limit is available for the multi-level queue. However, we prove the tightness and compute the distributional limit, not using any information on the corresponding process limit of \cite{AtarMiya2025} and the stationary distribution of \cite{Miya2024b}. These results are given in \thr{MLQ-main}, which is a main result of this paper. In particular, the limit of the stationary distribution for $K=2$ exactly corresponds to that in (ii) of Theorem 3.1 of \cite{Miya2025} (see \rem{MLQ-main 1}).

Thus, as for the heavy traffic limit of the stationary distribution, the BAR approach nicely works for the multi-level queue, but not fully for the 2-level $GI/G/1$ queue. From this observation, it is interesting to see whether the BAR approach of this paper is useful for more complicated queues including their networks which have state-dependent time changes for arrivals and services. They may be an interesting topic in future research.

This paper is made up by five sections. In \sectn{preliminaries}, we introduce the sequence of indexed multi-level queues together with notations, present the heavy traffic assumptions, and discuss a stationary framework including Palm distributions. Main result, \thr{MLQ-main}, is presented in \sectn{main}, and proved in \sectn{proofs}. Finally, the approach to prove the tightness not using the BAR is discussed in \sectn{concluding} with help of \app{lem-th1}.

\section{Preliminaries and assumptions}
\label{sec:preliminaries}
\setnewcounter

In this section, we present preliminary results and assumptions in four subsections. In Sections \sect{multi-level} and \sect{time-evolution}, we introduce the multi-level queue and describe its time evolution. We then present heavy traffic conditions for diffusion approximation in \sectn{HT-a}.

In those three subsections, we follow the notation system of \cite{AtarMiya2025} in principle. We also use standard notations such as $a \vee b = \max(a,b)$ and $a \wedge b = \min(a,b)$ for $a,b \in \dd{R}$, where $\dd{R}$ is the set of all real numbers. As usual convention, all processes are assumed to be right-continuous with left-limits unless stated otherwise.

However, there are two differences in our notations from those in \cite{AtarMiya2025}. First, we do not need the time scaling for diffusion approximation because we only consider the stationary distributions. Because of this, we do not include time scaling in system parameters as in \cite{AtarMiya2025} except for \sectn{concluding}, where process limits are considered. Secondly, we need a Markov process for describing the multi-level queue, which is not the case there. This Markov process will be used to characterize the stationary distribution. We denote the Markov and its queue length process by $X(\cdot)$ and $L(\cdot)$, respectively, while $X(\cdot)$ is used for the queue length process in \cite{AtarMiya2025}. These differences in mind, we introduce notations.

Finally, in \sectn{stationary}, we discuss the stationary framework for the diffusion approximation of the stationary distribution, following \cite{Miya2024,Miya2025}. Namely, we introduce a stochastic basis $(\Omega,\sr{F},\dd{F},\dd{P})$ and time-shift operator semigroup $\Theta_{\bullet} \equiv \{\Theta_{t}; t \ge 0\}$ on $\Omega$. Then, we define Palm distributions concerning counting processes following \cite{Miya2024}.

\subsection{Sequence of multi-level queues}
\label{sec:multi-level}

We index the sequence of the multi-level queues by positive integers $n$. These multi-level queues have the common number of levels independent of $n$, which is denoted by $K \ge 2$. Thus, these multi-level queues are also called $K$-level queues. Denote the levels of the $n$-th system by nonnegative real numbers $\ell^{(n)}_{0}, \ell^{(n)}_{1}, \ldots, \ell^{(n)}_{K-1}$ which satisfy that $\ell^{(n)}_{0} = 0 < \ell^{(n)}_{1} < \ldots < \ell^{(n)}_{K-1} < \infty$ and $\ell^{(n)}_{i} - \ell^{(n)}_{i-1} \ge 1$ for $i=1,2,\ldots, K-1$. Here, we can limit $\ell^{(n)}_{i}$'s to be integers, but we do not so because real numbers are more appearing when the queue length is scaled. By the same reason, we choose $\dd{R}_{+} = [0,\infty)$ as the common state space of the queue length processes. We partition this state space $\dd{R}_{+}$ as 
\begin{align*}
 & S^{(n)}_{1} = [0,\ell^{(n)}_{1}], \quad S^{(n)}_{i} = (\ell^{(n)}_{i-1}, \ell^{(n)}_{i}], \quad i =2, 3, \ldots, K-1, \quad S^{(n)}_{K} = (\ell^{(n)}_{K-1}, \infty).
\end{align*}
For convenience, we set $S^{(n)}_{0} = \{0\}$, and call $S^{(n)}_{i}$ the $i$-th level set. Note that the definition of $S^{(n)}_{1}$ is slightly different from that in \cite{AtarMiya2025}. We assume that each system has a single server and infinite buffer, and customers are served in the manner of first-come first-served, where the server is busy as long as the system is not empty. 

Following \cite{AtarMiya2025}, we introduce common counting processes $A(t)$ and $S(t)$ for the arrivals and services of the $n$-th system. Denote the $j$-th counting times of $A$ and $S$ by $t_{A,j}$ and $t_{S,j}$, respectively, for $j=1,2,\ldots$. Let $t_{A,0} = t_{S,0} = 0$, and let
\begin{align*}
  T_{A}(j) = t_{A,j+1} - t_{A,j}, \qquad T_{S}(j) = t_{S,j+1} - t_{S,j}, \qquad j=0, 1, \ldots.
\end{align*}
Then, we have
\begin{align*}
  A(t) = \min\left\{k \ge 0; t_{A,k} > t\right\}, \qquad S(t) = \min\left\{k \ge 0; t_{S,k} > t\right\}, \qquad t \ge 0.
\end{align*}

Assume the following conditions for the arrivals and the service times.

\begin{mylist}{3}
\item [(\sect{preliminaries}.a)] $A(\cdot)$ and $S(\cdot)$ are delayed renewal processes and independent of each other. Namely, $T_{A}(j) > 0$, $j=1,2,\ldots$ are $i.i.d.$ random variables, while $T_{A}(0) \ge 0$ is independent of $T_{A}(j)$ for $j \ge 1$. Similar independences are assumed for $T_{S}(j)$ for $j=0,1,\ldots$.

\item [(\sect{preliminaries}.b)]  Denote by $T_{A}$ and $T_{S}$ random variables subject to the common distributions of $T_{A}(j)$ and $T_{S}(j)$ for $j \ge 1$, respectively. These $T_{A}$ and $T_{S}$ have unit means and finite variances $\sigma_{A}^{2}$ and $\sigma_{S}^{2}$ such that $\sigma_{A}^{2} + \sigma_{S}^{2} > 0$, respectively, while $T_{A}(0)$ and $T_{S}(0)$ do not necessarily have the same distributions as $T_{A}$ and $T_{S}$, respectively, but have finite means and variances.
\end{mylist}

Using these renewal processes, we inductively define the queue length process $L^{(n)}(\cdot) \equiv \{L^{(n)}(t); t \ge 0\}$, arrival counting process $N^{(n)}_{e}(\cdot)$ and departure counting process $N^{(n)}_{d}(\cdot)$ of the $n$-th multi-level queue. 

Let $t^{(n)}_{e,j}$ be the $j$-th arrival time, and let $R^{(n)}_{e}(t)$ be the nominal time to the next arrival at time $t$, which has the following dynamics. Let $R^{(n)}_{e}(t^{(n)}_{e,j}) = T_{A}(j)$, and let $\lambda^{(n)}_{i} > 0$ be the speed of time for arrivals under $L^{(n)}(t) \in S^{(n)}_{i}$ for $i=1,2,\ldots,K$, namely,
\begin{align}
\label{eq:Re-t}
   \frac {d}{dt} R^{(n)}_{e}(t) = - \sum_{i=1}^{K} \lambda^{(n)}_{i} 1(L^{(n)}(t) \in S^{(n)}_{i}), \qquad t \ge t^{(n)}_{e,j},
\end{align}
as long as $\inf_{t^{(n)}_{e,j} < s \le t} R^{(n)}_{e}(s) > 0$, then the $j+1$ arrival time is defined as $t^{(n)}_{e,j+1} = \inf \{t > t^{(n)}_{e,j}; R^{(n)}_{e}(t-) = 0\}$. We let $\lambda^{(n)}_{0} = \lambda^{(n)}_{1}$, which is not assumed in \cite{AtarMiya2025} because $S^{(n)}_{1} = (0,\ell^{(n)}_{1}]$ there. However, this difference does not make any change for the process limit in \cite{AtarMiya2025}. Note that $R^{(n)}_{e}(t)$ is not the residual arrival time to the next arrival, but the amount of time to be processed until the next arrival under the time change. We call $R^{(n)}_{e}(t)$ the nominal time of residual arrival at $t$. Thus,  the counting process $N^{(n)}_{e}(\cdot)$ for arrivals is defined as
\begin{align}
\label{eq:Ne-t}
  N^{(n)}_{e}(t) = \min\{j \ge 0; t < t^{(n)}_{e,j+1}\}, \qquad t \ge 0.
\end{align}

Similarly, let $t^{(n)}_{d,j}$ be the $j$-th departure time, then the nominal residual departure time $R^{(n)}_{d}(t)$ is defined, where $R^{(n)}_{d}(t_{d,j}-) = 0$ and $R^{(n)}_{d}(t_{d,j}) = T_{S}(j)$, and \eq{Re-t} is changed to
\begin{align}
\label{eq:Rd-t}
   \frac {d}{dt} R^{(n)}_{d}(t) = - \sum_{i=1}^{K} \mu^{(n)}_{i} 1(L^{(n)}(t) \in S^{(n)}_{i}) + \mu^{(n)}_{1} 1(L^{(n)}(t) \in S^{(n)}_{0}), \qquad t \ge t^{(n)}_{d,j},
\end{align}
as long as $\inf_{t^{(n)}_{d,j} < s \le t} R^{(n)}_{d}(s) > 0$, where $\mu^{(n)}_{i} > 0$ be the service speed under $L^{(n)}(t)$ in $S^{(n)}_{1} \setminus S^{(n)}_{0}$ or in $S^{(n)}_{i}$ for $i = 2,3,\ldots,K$. Then, the counting process $N^{(n)}_{d}(\cdot)$ for departures is defined as
\begin{align}
\label{eq:Nd-t}
  N^{(n)}_{d}(t) = \min\{j \ge 0; t < t^{(n)}_{d,j+1}\}, \qquad t \ge 0.
\end{align}
Note that $R^{(n)}_{d}(t)$ is unchanged when $L^{(n)}(t) = 0$, that is, the system is empty.

Since the multi-level queue has a single server, we must have
\begin{align}
\label{eq:Ln-t}
  L^{(n)}(t) = L^{(n)}(0) + N^{(n)}_{e}(t) - N^{(n)}_{d}(t), \qquad t \ge 0.
\end{align}
Define processes $R^{(n)}(\cdot) \equiv \{R^{(n)}(t); t \ge 0\}$ and $X^{(n)}(\cdot) \equiv \{X^{(n)}(t); t \ge 0\}$ as
\begin{align*}
  R^{(n)}(t) = (R^{(n)}_{e}(t), R^{(n)}_{d}(t)), \qquad X^{(n)}(t) = (L^{(n)}(t), R^{(n)}_{e}(t), R^{(n)}_{d}(t)), \qquad t \ge 0,
\end{align*}
then it is not hard to see that $X^{(n)}(\cdot)$ is inductively constructed through $\{X^{(n)}(s); 0 \le s \le t^{(n)}_{j}\}$ for $j=0,1,\ldots$, where 
\begin{align*}
  t^{(n)}_{j} = \inf_{t \ge 0} \{t \ge 0; j \le N^{(n)}_{e}(t) + N^{(n)}_{d}(t) \}.
\end{align*}
Note that $t^{(n)}_{0} = 0$. Define the counting process $N^{(n)}(\cdot)$ generated by $t^{(n)}_{j}$'s as
\begin{align}
\label{eq:Nn-t}
  N^{(n)}(t) = \inf\{j \ge 0; t < t^{(n)}_{j+1}\}, \qquad t \ge 0.
\end{align}
Note that $X^{(n)}(t)$ is continuous in $t \ge 0$ except for the counting instants of $N^{(n)}(\cdot)$.

By the construction of $X^{(n)}(\cdot)$, it can be defined on the common stochastic basis $(\Omega,\sr{F},\dd{F},\dd{P})$ with the common time-shift operator $\Theta_{\bullet}$ because $\{X^{(n)}(\cdot); n \ge 1\}$ is the collection of countably many stochastic processes. Note that $X^{(n)}(t)$ is right-continuous with left-hand limits in $t \ge 0$, and $X^{(n)}(\cdot)$ can be defined as a strong Markov process with respect to the filtration $\dd{F}^{(n)} \equiv \{\sr{F}^{(n)}_{t}; t \ge 0\}$, where $\sr{F}^{(n)}_{t} = \sigma(\{X^{(n)}(s); s \in [0,t]\})$, that is, the $\sigma$-field generated by $X^{(n)}(s)$ for $s \in [0,t]$. Throughout the paper, we consider this strong Markov process.

\subsection{Time evolution of the $n$-th multi-level queue}
\label{sec:time-evolution}

Note that $X^{(n)}(\cdot)$ for the $K$-level queue is uniquely constructed by the following data.
\begin{align}
\label{eq:M-primitive}
  A, S, \{(\ell^{(n)}_{i}, \lambda^{(n)}_{i}, \mu^{(n)}_{i}); i=0,1,\ldots,K\},
\end{align}
where $\ell^{(n)}_{K} = \infty$ and $\mu^{(n)}_{0} = 0$. We call this data the modeling primitives of the $n$-th system. We next describe the time evolution of $X^{(n)}(\cdot)$ using a test function. For this, we introduce some notations. We choose $\dd{R}_{+}^{3}$ for the state space of $X^{(n)}(\cdot)$. Let $D(\dd{R}_{+}^{3})$ be the set of all measurable functions from $\dd{R}_{+}^{3}$ to $\dd{R}$ which have finitely many discontinuities at most in each finite time interval. Let $D_{b}(\dd{R}_{+}^{3})$ be the subset $D(\dd{R}_{+}^{3})$ whose elements are all bounded functions, and let $D_{b}^{p}(\dd{R}_{+}^{3})$ be the subset of $D_{b}(\dd{R}_{+}^{3})$ whose elements have bounded partial derivatives from the right concerning each variable. For $f \in D_{b}^{p}(\dd{R}_{+}^{3})$, denote the partial derivatives of $f(x_{1},x_{2},x_{3})$ concerning $x_{i}$ from the right by $\frac {\partial^{+}}{\partial x_{i}}$ for $i=1,2,3$.

Define operator $\sr{H}^{(n)}$ which maps $f \in D_{b}^{p}(\dd{R}_{+}^{3})$ to $\sr{H}^{(n)} f \in D_{b}(\dd{R}_{+}^{3})$ by
\begin{align*}
  & \sr{H}^{(n)} f(\vc{x}) = - \sum_{i = 1}^{K} \lambda^{(n)}_{i} 1(x_{1} \in S^{(n)}_{i}) \frac {\partial^{+}}{\partial x_{2}} f(\vc{x}) \nonumber\\
  & \quad  - \sum_{i = 1}^{K} \mu^{(n)}_{i} 1(x_{1} \in S^{(n)}_{i}) \frac {\partial^{+}}{\partial x_{3}} f(\vc{x}) + \mu^{(n)}_{1} 1(x_{1} \in S^{(n)}_{0}) \frac {\partial^{+}}{\partial x_{3}} f(\vc{x}), \quad \vc{x} \equiv (x_{1}, x_{2}, x_{3}) \in \dd{R}_{+}^{3}.
\end{align*}
This $\sr{H}^{(n)}$ describes the changes of $X^{(n)}(t)$ between the adjacent counting instants of $N^{(n)}(\cdot)$.

We also introduce a difference operator $\Delta$ to describe the discontinuous changes of $X^{(n)}(t)$ at counting instants of $N^{(n)}(\cdot)$. Define it as $\Delta g(t) = g(t) - g(t-)$ for function $g$ from $\dd{R}_{+}$ to $\dd{R}$ which is right-continuous with left-limits. At these counting instants, arrival and service completion may simultaneously occur. To describe them precisely, we assume that an arrival occurs first, and introduce an intermediate state between state changes by arrival and departure.

Thus, we define $\dd{F}$-adapted process $X^{(n)}_{1}(\cdot) \equiv \{(L^{(n)}_{1}(t), R^{(n)}_{e,1}(t), R^{(n)}_{d,1}(t)); t \ge 0\}$ as 
\begin{align*}
  & X^{(n)}_{1}(t) = 
\begin{cases}
X^{(n)}(t-), & \Delta N^{(n)}_{e}(t) = 0,\\
 (L^{(n)}(t-) + 1,  T_{A}(N_{e}(t)), R^{(n)}_{d}(t-)), & \Delta N^{(n)}_{e}(t) = 1,
\end{cases}
\end{align*}
and call $X^{(n)}_{1}(t)$ an intermediate state. Then, just after time $t$, we have
\begin{align*}
  X^{(n)}(t) = \begin{cases}
(L^{(n)}_{1}(t) -1, R^{(n)}_{e,1}(t), T_{S}(N^{(n)}_{d}(t))), & \Delta N^{(n)}_{d}(t) = 1\\
X^{(n)}_{1}(t), & \Delta N^{(n)}_{d}(t) = 0.   
\end{cases}
\end{align*}
Thus, we can see that the arrival first assumption has no influence on the queue length. 

For convenience, we define another difference operators $\Delta_{e}$ and $\Delta_{d}$ as
\begin{align*}
 & \Delta_{e} f(X^{(n)})(t) = f(X^{(n)}_{1}(t)) - f(X^{(n)}(t-)),\\
 & \Delta_{d} f(X^{(n)})(t) = f(X^{(n)}(t)) - f(X^{(n)}_{1}(t)),
\end{align*}
where function $f(X^{(n)})$ is defined by $f(X^{(n)})(t) = f((X^{(n)}(t))$ for $t \ge 0$. Recall the definition \eq{Nn-t} of $N^{(n)}(\cdot)$, and assume that
\begin{align}
\label{eq:test-f}
  \mbox{$f(X^{(n)})(t)$ is continuous in $t \ge 0$ if $\Delta N^{(n)}(t) = 0$.}
\end{align}
Then, $\Delta f(X^{(n)})(t) = \Delta_{e} f(X^{(n)})(t) + \Delta_{d} f(X^{(n)})(t)$, so, from the elementary observation of the time evolution of $X^{(n)}(t)$, we have
\begin{align}
\label{eq:fXn-t}
  f(X^{(n)}(t)) & = f(X^{(n)}(0)) + \int_{0}^{t} \sr{H}^{(n)} f(X^{(n)})(u) du \nonumber\\
  & \quad + \sum_{u \in (0,t]} [\Delta_{e} f(X^{(n)})(u) + \Delta_{d} f(X^{(n)})(u)], \qquad f \in D_{b}^{p}(\dd{R}_{+}^{3}),
\end{align}
where the summation is well defined because $\Delta_{e} f(X^{(n)})(u)$ and $\Delta_{d} f(X^{(n)})(u)$ vanish except for finitely many $u$ in $(0,t]$. 

\subsection{Heavy traffic assumptions}
\label{sec:HT-a}

We introduce heavy traffic conditions for the sequence of the multi-level queues indexed by $n$. We first recall some of the modeling primitives of the $n$-th system. 
\begin{align*}
 & \ell^{(n)}_{i}, \quad \lambda^{(n)}_{i}, \quad \mu^{(n)}_{i}, \qquad i=0,1,\ldots, K, \\
 & \dd{E}[T_{A}] = 1, \quad \sigma_{A}^{2} = \dd{E}[(T_{A} - 1)^{2}], \qquad \dd{E}[T_{S}]= 1, \quad \sigma_{S}^{2} = \dd{E}[(T_{S} - 1)^{2}].
\end{align*}
Using these notations, we define
\begin{align*}
   & \rho^{(n)}_{i} = \lambda^{(n)}_{i}/\mu^{(n)}_{i}, \qquad \sigma^{(n)}_{i} = (\lambda^{(n)}_{i} \sigma_{A}^{2} + \mu^{(n)}_{i} \sigma_{S}^{2})^{1/2}, \qquad i=1,2, \ldots, K,
\end{align*}
where $\sigma^{(n)}_{i} > 0$ by (\sect{preliminaries}.b) in \sectn{multi-level} and (\sect{preliminaries}.c) below.

We aim for $X^{(n)}(\cdot)$ to represent heavy traffic as $n \to \infty$. For this, we assume the following three conditions as $n \to \infty$.
\begin{mylist}{3}
\item [(\sect{preliminaries}.c)] There are $\lambda_{i} > 0, \; \widehat{\lambda}_{i} \in \dd{R}$ for $i=0,1, \ldots, K$ and $\mu_{i} > 0, \; \widehat{\mu}_{i} \in \dd{R}$ for $i=1,2, \ldots, K$ such that
\begin{align*}
  \lambda^{(n)}_{i} = \lambda_{i} + \widehat{\lambda}_{i} n^{-1/2} + o(n^{-1/2}) > 0, \qquad \mu^{(n)}_{i} = \mu_{i} + \widehat{\mu}_{i} n^{-1/2} + o(n^{-1/2}) > 0.
\end{align*}

\item [(\sect{preliminaries}.d)] For each $i=1,2, \ldots, K$, let $b_{i} = \widehat{\lambda}_{i} - \widehat{\mu}_{i}$, then $\lambda^{(n)}_{i} - \mu^{(n)}_{i} = b_{i} n^{-1/2} + o(n^{-1/2})$.

\item [(\sect{preliminaries}.e)] There are constants $\ell_{1}, \ell_{2}, \ldots, \ell_{K}$ such that $0 < \ell_{1} < \ell_{2}< \ldots < \ell_{K}$ and $\ell^{(n)}_{i} = \ell_{i} n^{1/2} + o(1)$ for $i=1,2,\ldots,K$.
\end{mylist}
Here, we write $f(n) = o(g(n))$ as $n \to \infty$ for positive valued functions $f, g$ of $n$ if \linebreak $\lim_{n \to \infty} f(n)/g(n) = 0$. Similarly, $f(n) = O(g(n))$ if $\limsup_{n \to \infty} f(n)/g(n) < \infty$.
Note that (\sect{preliminaries}.c) and (\sect{preliminaries}.d) imply that
\begin{align}
\label{eq:la-mu-limit}
  \lambda^{(n)}_{i} \to \lambda_{i}, \quad \mu^{(n)}_{i} \to \mu_{i}, \quad n \to \infty, \qquad \lambda_{i} = \mu_{i}, \qquad i=1,2,\ldots,K.
\end{align}

Let $\sigma_{i} = (\lambda_{i} \sigma_{A}^{2} + \mu_{i} \sigma_{S}^{2})^{1/2}$, and let
\begin{align*}
   & S_{1} = [0,\ell_{1}], \qquad S_{i} = (\ell_{i-1}, \ell_{i}], \quad i =2, 3, \ldots, K-1, \qquad S_{K} = (\ell_{K-1}, \infty).
\end{align*}
For the notational convenience, we introduce the following functions of $x \in \dd{R}_{+}$.
\begin{align}
\label{eq:b-sigma-x}
 & b(x) \equiv \sum_{i = 1}^{K} b_{i} 1(x \in S_{i}), \qquad \sigma(x) \equiv \sum_{i = 1}^{K} \sigma_{i} 1(x \in S_{i}).
\end{align}
Then, from (\sect{preliminaries}.c), (\sect{preliminaries}.d) and \eq{la-mu-limit}, 
\begin{align}
\label{eq:b-sigma-lim}
\begin{split}
 & \lim_{n \to \infty} n^{1/2} \sum_{i = 1}^{K} (\lambda^{(n)}_{i} - \mu^{(n)}_{i}) 1(n^{1/2} x \in S^{(n)}_{i}) = b(x), \\
 & \lim_{n \to \infty} \sum_{i =1}^{K} \sigma^{(n)}_{i} 1(n^{1/2} x \in S^{(n)}_{i}) = \sigma(x).
 \end{split}
\end{align}

\subsection{Stationary framework and Palm distributions}
\label{sec:stationary}

Our main interest is the stationary distributions of the indexed systems and its diffusion scaled limit in heavy traffic. For this, we need a condition for their stability, which is given below.
\begin{itemize}
\item [(\sect{preliminaries}.f)] $\rho^{(n)}_{K} < 1$ for all $n \ge 1$ and $b_{K} < 0$.				
\end{itemize}
We assume this stability condition throughout the paper from now on. Then, for each $n \ge 1$, the Markov process $X^{(n)}(\cdot)$ has the stationary distribution. Hence, we can make it a stationary process. In what follows, we consider this $X^{(n)}(\cdot)$ as a stationary strong Markov process for all $n \ge 1$ unless otherwise stated. Denote its stationary distribution by $\pi^{(n)}$, and let $\dd{P}_{\pi^{(n)}}$ be the probability measure on $(\Omega,\sr{F}^{(n)}_{\infty})$ such that
\begin{align}
\label{eq:Pn-stationary}
  \dd{P}_{\pi^{(n)}}(\Theta_{t}^{-1} A) = \dd{P}_{\pi^{(n)}}(A), \qquad A \in \sr{F}^{(n)}_{\infty}, t \ge 0,
\end{align}
where $\sr{F}^{(n)}_{\infty} = \sigma(\cup_{t \ge 0} \sr{F}^{(n)}_{t})$. Since we only consider the sequence of $X^{(n)}(\cdot)$ for $n \ge 1$, we can assume that
\begin{align}
\label{eq:P-stationary}
  \dd{P}(\Theta_{t}^{-1} A) = \dd{P}(A), \qquad A \in \sr{F}, t \ge 0,
\end{align}
appropriately choosing $(\Omega,\sr{F},\dd{F})$, and consider $\dd{P}_{\pi^{(n)}}$ as the restriction of $\dd{P}$ on $(\Omega,\sr{F}^{(n)}_{\infty})$. In this setting, we denote the expectation of a random variable $X$ generated from $X^{(n)}(\cdot)$ simply by $\dd{E}(X)$, instead of $\dd{E}_{\pi^{(n)}}(X)$.

Thus, we have the stationary framework on the stochastic basis $(\Omega,\sr{F},\dd{F},\dd{P})$. On this framework, $X^{(n)}(\cdot)$ is a stationary process consistent with time-shift $\Theta_{\bullet}$. Then, counting processes $N^{(n)}_{e}(\cdot)$ and $N^{(n)}_{d}(\cdot)$ are also stationary processes consistent with $\Theta_{\bullet}$ from their constructions by \eq{Re-t}, \eq{Ne-t}, \eq{Rd-t} and \eq{Nd-t}. Let
\begin{align*}
  \alpha^{(n)}_{e} = \dd{E}[N^{(n)}_{e}(1)], \qquad \alpha^{(n)}_{d} = \dd{E}[N^{(n)}_{d}(1)].
\end{align*}
We note the following simple facts.
\begin{lemma}
\label{lem:alpha 1}
$\alpha^{(n)}_{e}$ and $\alpha^{(n)}_{d}$ are finite.
\end{lemma}

\begin{proof}
It follows from $R^{(n)}_{e}(t^{(n)}_{e,j}) =  T_{A}(j)$ and \eq{Re-t} that
\begin{align*}
  0 & = R^{(n)}_{e}(t^{(n)}_{e,j+1}-) = R^{(n)}_{e}(t^{(n)}_{e,j}) + \int_{t^{(n)}_{e,j}}^{t^{(n)}_{e,j+1}} (R^{(n)}_{e})'(u) du\\
  & =  T_{A}(j) - \int_{t^{(n)}_{e,j}}^{t^{(n)}_{e,j+1}} \sum_{i=1}^{K} \lambda^{(n)}_{i} 1(L^{(n)}(u) \in S^{(n)}_{i}) du.
\end{align*}
This implies that
\begin{align*}
  t^{(n)}_{e,j+1} - t^{(n)}_{e,j} & = \inf\left\{t > 0; \sum_{i=1}^{K} \lambda^{(n)}_{i} \int_{t^{(n)}_{e,j}}^{t^{(n)}_{e,j}+t} 1(L^{(n)}(u) \in S^{(n)}_{i}) du \ge T_{A}(j)\right\} \\
  & \le (\max_{0 \le i \le K} \lambda^{(n)}_{i})^{-1} T_{A}(j).
\end{align*}
Since $N^{(n)}_{e}(\cdot)$ is a stationary increasing counting process, we have, by the law of large numbers,
\begin{align*}
  \dd{E}[N^{(n)}_{e}(1)] & = \lim_{t \to \infty} N^{(n)}(t)/t = \lim_{j \to \infty} N^{(n)}(t^{(n)}_{e,j})/t^{(n)}_{e,j} = \lim_{j \to \infty} j/t^{(n)}_{e,j} \le (\max_{0 \le i \le K} \lambda^{(n)}_{i})^{-1} \dd{E}[T_{A}].
\end{align*}
Hence, $\alpha^{(n)}_{e} = \dd{E}[N^{(n)}_{e}(1)] < \infty$. The finiteness of $\alpha^{(n)}_{d}$ is similarly proved.
\end{proof}

Since $\alpha^{(n)}_{e}$ and $\alpha^{(n)}_{d}$ are finite, we can define probability distributions $\dd{P}^{(n)}_{e}$ and $\dd{P}^{(n)}_{d}$ on $(\Omega,\sr{F})$ by
\begin{align}
\label{eq:Palm 1}
  \dd{P}^{(n)}_{v}(A) = (\alpha^{(n)}_{v})^{-1} \dd{E}\left[ \int_{(0,1]} 1_{\Theta_{u}^{-1} A} N^{(n)}_{v}(du) \right], \qquad A \in \sr{F}, v=e,d,
\end{align}
where $1_{A}$ is the indicator function of the event $A \in \sr{F}$. For $v=e,d$, $\dd{P}^{(n)}_{v}$ is called a Palm distribution concerning the counting process $N^{(n)}_{v}$. The intuitive meaning of the Palm distributions is briefly explained in \cite{Miya2025} (see \cite{Miya2024} for their details).

We use the Palm distributions for the derivation of the stationary equations as shown below. Take the expectation of \eq{fXn-t} for $t=1$ by $\dd{P}$, then the stationarity of $X^{(n)}(\cdot)$ and the definition of the Palm distribution yield
\begin{align*}
 & \dd{E}\left[\int_{0}^{1} \sr{H}^{(n)} f(X^{(n)}(u)) du\right] = \int_{0}^{1} \dd{E}\left[\sr{H}^{(n)} f(X^{(n)}(u))\right] du = \dd{E}\left[\sr{H}^{(n)} f(X^{(n)}(0))\right],\\
 & \dd{E}\left[\sum_{u \in (0,1]} \Delta_{v} f(X^{(n)})(u)\right] = \alpha^{(n)}_{v} \dd{E}^{(n)}_{v}\left[\Delta_{v} f(X^{(n)})(0) \right], \qquad v= e,d,
\end{align*}
where $\dd{E}^{(n)}_{v}$ stands for the expectation under $\dd{P}^{(n)}_{v}$ for $v=e,d$. Hence, if the assumptions (\sect{preliminaries}.a), (\sect{preliminaries}.b) and (\sect{preliminaries}.f) are satisfied, then we have a stationary equation:
\begin{align}
\label{eq:BAR1}
  & \dd{E}\left[ \sr{H}^{(n)}f(X^{(n)}(0)) \right]  \nonumber\\
  & \quad + \alpha^{(n)}_{e} \dd{E}^{(n)}_{e}\left[\Delta_{e} f(X^{(n)})(0) \right] + \alpha^{(n)}_{d} \dd{E}^{(n)}_{d}\left[\Delta_{d} f(X^{(n)})(0) \right] = 0, \qquad f \in D_{b}^{p}(\dd{R}_{+}^{3}).
\end{align}

Note that $\dd{E}\left[ \sr{H}^{(n)}f(X^{(n)}(0)) \right]$ can be written as $\dd{E}\left[ \sr{H}^{(n)}f(X^{(n)}) \right]$ by our notational convention. The formula \eq{BAR1} is a special case of the rate conservation law (e.g., see \cite{Miya1994}), and referred to as a basic adjoint relationship, BAR in short.

\section{Main results}
\label{sec:main}
\setnewcounter

We are now ready to present a main result, which will be proved in \sectn{proofs}. For convenience, by $X^{(n)} \equiv (L^{(n)}, R^{(n)}_{e}, R^{(n)}_{d})$, we denote a random vector subject to the stationary distribution of the Markov process $X^{(n)}(\cdot)$ for the $n$-th $K$-level $GI/G/1$ queue. 

Define probability distribution $\nu^{(n)}$ under the assumptions (\sect{preliminaries}.a)--(\sect{preliminaries}.f) as
\begin{align}
\label{eq:hnu-B}
  \nu^{(n)}(B) = \dd{P}(n^{-1/2}L^{(n)} \in B), \qquad B \in \sr{B}(\dd{R}_{+}).
\end{align}
We refer to this $\nu^{(n)}$ as a scaled stationary distribution. For convenience, we denote the set $\{1,2,\ldots,j\}$ by $J_{j}$ for $j=1,2,\ldots,K$. Let 
\begin{align*}
  d^{(n)}_{i} = \dd{P}(L^{(n)} \in S^{(n)}_{i}), \qquad i \in J_{K},
\end{align*}
and define the conditional distributions of $\nu^{(n)}$ on $S^{(n)}_{i}$ by
\begin{align*}
 \nu^{(n)}_{i}(B) = \dd{P}(n^{-1/2}L^{(n)} \in B| L^{(n)} \in S^{(n)}_{i}), \qquad B \in \sr{B}(\dd{R}_{+}), \; i \in J_{K},
\end{align*}
as long as $d^{(n)}_{i} > 0$. For convenience, we put $\nu^{(n)}_{i}$ to be null when $d^{(n)}_{i} = 0$, then we have
\begin{align*}
  \nu^{(n)}(B) = \sum_{i \in J_{K}} d^{(n)}_{i} \nu^{(n)}_{i}(B), \qquad B \in \sr{B}(\dd{R}_{+}).
\end{align*}

We also need the following notations.
\begin{align}
\label{eq:MLQ-beta 1}
  \beta_{i} = \frac {2 b_{i}} {\sigma_{i}^{2}}, \qquad i \in J_{K},
\end{align}
where recall that $\sigma_{i}^{2} = \lambda_{i} \sigma_{A}^{2} + \mu_{i} \sigma_{S}^{2}$. Define $\xi_{j}$ as
\begin{align}
\label{eq:xi}
  \xi_{0} = 1, \qquad \xi_{j} = \prod_{i=1}^{j} e^{\beta_{i}(\ell_{i} - \ell_{i-1})}, \quad j \in J_{K-1}, \qquad \xi_{K} = 0,
\end{align}
and let
\begin{align}
\label{eq:ci}
  c_{i} = \begin{cases}
 2\sigma_{i}^{-2} (\ell_{i-1} - \ell_{i}) \xi_{i-1}, & b_{i} = 0,\\
 b_{i}^{-1} (1 - e^{\beta_{i} (\ell_{i} - \ell_{i-1})}) \xi_{i-1} \; & b_{i} \not= 0,
\end{cases}
 \quad i \in J_{K-1}, \qquad c_{K} = b_{K}^{-1} \xi_{K-1},
\end{align}
where it is noticed that $c_{i} < 0$.

\begin{theorem}
\label{thr:MLQ-main}
Assume the conditions (\sect{preliminaries}.a)--(\sect{preliminaries}.f), and, for $i \in J_{K}$, let $\nu_{i}$ be the probability distribution on $(\dd{R}_{+}, \sr{B}(\dd{R}_{+}))$ which has the density function $h_{i}$:
\begin{align}
\label{eq:hj}
 & h_{i}(x) = \begin{cases}
 \ds \left(\frac 1{\ell_{i} - \ell_{i-1}} 1(b_{i} = 0) + \frac { \beta_{i} e^{\beta_{i} (x - \ell_{i-1})}} {e^{\beta_{i} (\ell_{i} - \ell_{i-1})} - 1}1(b_{i} \not= 0)\right) 1(x \in S_{i}), & i \in J_{K-1}, \\
 -\beta_{K} e^{\beta_{K} (x - \ell_{K-1})} 1(x \in S_{K}), & i = K, 
\end{cases}
\end{align}
and let $\nu$ be the probability distribution on $\sr{B}(\dd{R}_{+})$ whose density function $h$ is given by
\begin{align}
\label{eq:h}
  h(x) = \sum_{i=1}^{K} d_{i} h_{i}(x), \qquad x \ge 0,
\end{align}
where 
\begin{align}
\label{eq:di}
 & d_{i} = 
 \frac {c_{i}} {\sum_{j=1}^{K} c_{j}}, \qquad i \in J_{K}.
\end{align}
Then,
\begin{align}
\label{eq:h-nu-lim}
  \nu^{(n)}_{i} \Rightarrow \nu_{i}, \quad i \in J_{K}, \qquad \nu^{(n)} \Rightarrow \nu, \qquad n \to \infty,
\end{align}
where ``$\Rightarrow$'' stands for the weak convergence of a sequence of distributions.
\end{theorem}

\begin{remark}
\label{rem:MLQ-main 1}
The stationary distribution $\nu$ of this theorem for $K=2$ exactly corresponds to the stationary distribution of the 2-level queue obtained in (ii) of Theorem 3.1 of \cite{Miya2025} if $b_{i}$ and $\sigma_{i}^{2}$ are replaced, respectively, by $-\mu b_{i}$ and $\mu c_{i} \sigma_{i}^{2}$ for the notations $b_{i}, c_{i}$ and $\sigma_{i}^{2}$ of \cite{Miya2025}. This supports the conjecture in Remark 3.2 of \cite{Miya2025} that the stationary distribution in (ii) holds without any additional assumption.
\end{remark}

The following corollary is immediate from Corollary 3,1 of \cite{Miya2024b}.

\begin{corollary}
\label{cor:MLQ-main}
Assume the assumptions of \thr{MLQ-main} and let
\begin{align}
\label{eq:MLQ-beta-x 1}
\beta(x) = \frac {2b(x)} {\sigma^{2}(x)}, \quad x \in \dd{R}_{+},
\end{align}
where recall the definition \eq{b-sigma-x} for $b(x)$ and $\sigma(x)$. Then, the probability distribution $\nu$ of \thr{MLQ-main} has the density function $h$ given by
\begin{align}
\label{eq:MLQ-h}
  h(x) = \frac 1{C \sigma^{2}(x)} \exp\left(\int_{0}^{x} \beta(y) dy\right), \qquad x \in \dd{R}_{+},
\end{align}
where $C = \int_{0}^{\infty} h(x) dx$, that is, $C$ is the normalizing constant. 
\end{corollary}

\begin{remark}
\label{rem:MLQ-main 2}
As shown in Theorem 3.2 of \cite{Miya2024b}, $\nu$ is the stationary distribution of the reflected diffusion $Z(\cdot)$ on $[0,\infty)$ which is obtained as the unique solution of the following stochastic integral equation.
\begin{align}
\label{eq:SIE-Z}
  Z(t) = Z(0) + \int_{0}^{t} b(Z(s)) ds + \int_{0}^{t} \sigma(Z(s)) dW(s) + Y(t) \ge 0, \qquad t \ge 0,
\end{align}
where $W(\cdot)$ is the one-dimensional standard Brownian motion and $Y(\cdot)$ is the nondecreasing process satisfying that $\int_{0}^{t} Z(t) dY(t) = 0$ for $t \ge 0$. By Theorem 2.1 of \cite{AtarMiya2025}, this reflected diffusion $Z(\cdot)$ is obtained as the weak limit of the diffusion scaled queue length process of the multi-level queue. Hence, \thr{MLQ-main} shows the limit interchange of the sequences of processes and stationary distributions.
\end{remark}

\section{Proof of \thr{MLQ-main}}
\label{sec:proofs}
\setnewcounter

\thr{MLQ-main} can be proved if we show that $\{\nu^{(n)}; n \ge 1\}$ is tight because the process limit is obtained in Theorem 2.1 of \cite{AtarMiya2025} (see \sectn{concluding} for details). However, we prove \thr{MLQ-main} neither using Theorem 2.1 of \cite{AtarMiya2025} nor using Theorem 3.2 of \cite{Miya2024b} (see \rem{MLQ-main 2}). Thus, the proof here is independent of the proofs of those theorems. 

We prove \thr{MLQ-main} in three steps, following the method used for the proof of Theorem 3.1 of \cite{Miya2025}. The first step considers some of basic notions and facts, the second step proves auxiliary lemmas for asymptotic expansions of several quantities, and the third step computes the limiting distribution, which proves the tightness. These three steps are given in Sections \sect{basic}, \sect{asymptotic} and \sect{computing}, respectively.

\subsection{Basic facts}
\label{sec:basic}

We consider some properties of $\alpha^{(n)}_{e}$ and $L^{(n)}(\cdot)$ under the Palm distributions. Similar to Lemma 4.1 of \cite{Miya2025}, we have the following lemma. Recall that all the sample paths are right-continuous with left-limits, and $X^{(n)} = (L^{(n)}, R^{(n)}_{e}, R^{(n)}_{d})$ is a random vector subject to the stationary distribution of $X^{(n)}(\cdot)$.
\begin{lemma}
\label{lem:MLQ-basic 1}
Under conditions (\sect{preliminaries}.a), (\sect{preliminaries}.b) and (\sect{preliminaries}.f), $\alpha^{(n)}_{e} = \alpha^{(n)}_{d}$, and
\begin{align}
\label{eq:MLQ-balance 1}
 & \dd{P}^{(n)}_{e}[L^{(n)}(0-) = \ell] = \dd{P}^{(n)}_{d}[L^{(n)}(0) = \ell], \qquad \ell \ge 0,\\
\label{eq:MLQ-alpha 1}
 & \alpha^{(n)}_{e} = \sum_{i = 1}^{K} \lambda^{(n)}_{i} \dd{P}\left[L^{(n)} \in S^{(n)}_{i}\right],\\
\label{eq:MLQ-alpha 2}
 & \alpha^{(n)}_{d} = \mu^{(n)}_{1} \dd{P}\left[0 < L^{(n)} \le \ell_{1} \right] + \sum_{i = 2}^{K} \mu^{(n)}_{i} \dd{P}\left[L^{(n)} \in S^{(n)}_{i}\right].
\end{align}
\end{lemma}

\begin{proof}
For each integer $\ell \ge 0$, let $f(\vc{x}) = x_{1} \wedge (\ell+1)$ for $\vc{x} \in \dd{R}_{+}^{3}$, then obviously $f \in D_{b}^{p}(\dd{R}_{+}^{3})$. Applying this $f$ to \eq{BAR1}, we have 
\begin{align}
\label{eq:MLQ-alpha}
  \alpha^{(n)}_{e} \dd{P}^{(n)}_{e}[(L^{(n)}(0-) \le \ell] = \alpha^{(n)}_{d} \dd{P}^{(n)}_{d}[L^{(n)}(0) \le \ell], \qquad \ell \ge 0.
\end{align}
because $\sr{H}^{(n)}f(X^{(n)}(0)) = 0$ and
\begin{align*}
 \Delta_{e} f(X^{(n)})(0) & = L_{1}^{(n)}(0) \wedge (\ell+1) - L^{(n)}(0-) \wedge (\ell+1) \\
 & = 1(L^{(n)}(0-) \le \ell, \Delta N^{(n)}_{e}(0) = 1),\\
 \Delta_{d} f(X^{(n)})(0) & = L^{(n)}(0) \wedge (\ell+1) - L^{(n)}_{1}(0) \wedge (\ell+1) \\
 & = - 1(L^{(n)}(0) \le \ell, \Delta N^{(n)}_{d}(0) = 1).
\end{align*}
Letting $\ell \to \infty$ in \eq{MLQ-alpha}, we have $\alpha^{(n)}_{e} = \alpha^{(n)}_{d}$. Applying this fact to \eq{MLQ-alpha}, we have \eq{MLQ-balance 1}. We next let $f(\vc{x}) = x_{2}$, and apply it to \eq{BAR1}, then
\begin{align*}
   - \sum_{i = 1}^{K} \lambda^{(n)}_{i} \dd{P}\left[L^{(n)}(0) \in S^{(n)}_{i}\right] + \alpha^{(n)}_{e} = 0
\end{align*}
because $\dd{E}[T_{A}] = 1$. Hence, we have \eq{MLQ-alpha 1} because $L^{(n)}(0)$ has the same distribution as $L^{(n)}$ under $\dd{P}$. Similarly, we have \eq{MLQ-alpha 2} from \eq{BAR1} for $f(\vc{x}) = x_{3}$.
\end{proof}

Since it follows from \eq{MLQ-alpha 1} and \eq{MLQ-alpha 2} that
\begin{align}
\label{eq:alpha-r-bounds}
  0 < \min_{i =1,2,\ldots,K} \lambda^{(n)}_{i} \le \alpha^{(n)}_{e} = \alpha^{(n)}_{d} \le \max_{i =1,2,\ldots,K} \mu^{(n)}_{i}, \qquad n \ge 1,
\end{align}
we have the following corollary.
\begin{corollary}
\label{cor:alpha 2}
$\alpha^{(n)}_{e}$ is uniformly bounded away from $0$ and above for all $n \ge 1$.
\end{corollary}

In what follows, we always replace $\alpha^{(n)}_{d}$ by $\alpha^{(n)}_{e}$, which is possible by \lem{MLQ-basic 1}. This simplifies the BAR \eq{BAR1}.

\subsection{Asymptotic expansions}
\label{sec:asymptotic}

 We apply an exponential type of a test function to the BAR  \eq{BAR1}, and compute various quantities under the stationary distribution $\dd{P}$ and the Palm distributions $\dd{P}^{(n)}_{e}$ and $\dd{P}^{(n)}_{d}$ introduced in \sectn{stationary}. For this, we prepare notations and lemmas.
 
Let $\varphi^{(n)}_{i}$ be the moment generating functions of probability measure $\nu^{(n)}_{i}$ for $i \in J_{K}$. Namely, 
\begin{align*}
 & \varphi^{(n)}_{i}(\theta_{i}) = \dd{E}[e^{\theta_{i} n^{-1/2}L^{(n)}}|L^{(n)} \in S^{(n)}_{i}],
\end{align*}
where $\theta_{i} \in \dd{R}$ for $i \in J_{K-1}$ and $\theta_{i} \le 0$ for $i=K$. To compute these moment generating functions, we will use the BAR \eq{BAR1} for test functions of exponential type, which are typically used in the BAR approach (e.g, see \cite{BravDaiMiya2017,Miya2017}). We first define function $g^{(n)}_{\theta}$ for $\theta \in \dd{R}$ as
\begin{align*}
 & g^{(n)}_{\theta}(\vc{y}) = e^{ - \eta^{(n)}(\theta) (y_{1} \wedge n^{1/2}) - \zeta^{(n)}(\theta) (y_{2} \wedge n^{1/2})}, \qquad \vc{y}= (y_{1},y_{2}) \in \dd{R}_{+}^{2},
\end{align*}
where $\eta^{(n)}(\theta), \zeta^{(n)}(\theta)$ are the solutions of the following equations for each $\theta \in \dd{R}$. 
\begin{align}
\label{eq:MLQ-boundary1}
  e^{\theta} \dd{E}\left[e^{-\eta^{(n)}(\theta) (T_{A} \wedge n^{1/2})}\right] = 1, \qquad e^{-\theta} \dd{E}\left[e^{-\zeta^{(n)}(\theta) (T_{S} \wedge n^{1/2})}\right] = 1.
\end{align}
These solutions exist and are finite because $\eta^{(n)}(\theta)$ is the inverse function of the Laplace transform of the finite positive random variable $T_{A} \wedge n^{1/2}$ which takes the value $e^{-\theta}$ and $\zeta^{(n)}(\theta)$ is similarly determined. Furthermore, $\eta^{(n)}(\theta)$ and $\zeta^{(n)}(\theta)$ are infinitely many differentiable functions of $\theta$. Hence, twice differentiating the equations in \eq{MLQ-boundary1}, the following lemma can be obtained by their Taylor expansions around the origin. Its proof can be found in \cite[Lemma 5.8]{BravDaiMiya2024}.
\begin{lemma}
\label{lem:MLQ-eta-zeta1}
Under the assumptions (\sect{preliminaries}.a)--(\sect{preliminaries}.e), we have the following expansions as $n^{-1/2} \theta \to 0$.
\begin{align}
\label{eq:MLQ-eta1}
 & \eta^{(n)}(n^{-1/2}\theta) = n^{-1/2} \theta + \frac 12 \sigma_{A}^{2} n^{-1} \theta^{2} + o(n^{-1} \theta^{2}),\\
\label{eq:MLQ-zeta1}
 & \zeta^{(n)}(n^{-1/2}\theta) = - n^{-1/2} \theta + \frac 12 \sigma_{S}^{2} n^{-1} \theta^{2} + o(n^{-1} \theta^{2}).
\end{align}
Furthermore, there are constants $d_{A}, d_{S}, a > 0$ such that, for $n \ge 1$ and $\theta \in \dd{R}$ satisfying $n^{-1/2}|\theta| < a$,
\begin{align}
\label{eq:MLQ-eta-zeta1}
 \big|\eta^{(n)}(n^{-1/2}\theta) (u_{1} \wedge n^{1/2}) & + \zeta^{(n)}(n^{-1/2}\theta) (u_{2} \wedge n^{1/2}) \big| \nonumber\\
  & \le |\theta| \left(d_{A} (n^{-1/2} u_{1} \wedge 1) + d_{S} (n^{-1/2}u_{2} \wedge 1)\right).
\end{align}
\end{lemma}

By this lemma, $g^{(n)}_{n^{-1/2}\theta}(R^{(n)})$ is bounded by $e^{|\theta|(d_{A} + d_{S})}$ for sufficiently large $n$ for each $\theta \in \dd{R}$, where recall that $R^{(n)} = (R^{(n)}_{e},R^{(n)}_{d})$. Hence, we have the following facts.

\begin{lemma}
\label{lem:MLQ-gn-limit}
$g^{(n)}_{n^{-1/2}\theta}(R^{(n)})$ converges to $1$ as $n \to \infty$ for each $\theta \in \dd{R}$, and
\begin{align}
\label{eq:MLQ-Egn-limit}
 & \lim_{n \to \infty} \dd{E}\left[g^{(n)}_{n^{-1/2}\theta}(R^{(n)})\right] = 1, \qquad \lim_{n \to \infty} \dd{E}_{u}\left[g^{(n)}_{n^{-1/2}\theta}(R^{(n)})\right] = 1, \quad u = e, d.
\end{align}
\end{lemma}

Recall that Markov process $X^{(n)}(\cdot)$ describes \ the $n$-th pre-limit process of the multi-level queue. Our next task is to compute a BAR for this process under suitable scaling, and derive its asymptotic version as $n \to \infty$. To this end, we define test function $f^{(n)}_{\vc{\theta}}$ for the BAR and $\vc{\theta} \equiv (\theta_{1}, \theta_{2}, \ldots, \theta_{K}) \in \dd{R}^{K-1} \times \dd{R}_{-}$ as
\begin{align}
\label{eq:f-theta-r}
 & f^{(n)}_{\vc{\theta}}(\vc{x}) = \sum_{i =1}^{K} e^{\theta_{i} x_{1}} g^{(n)}_{\theta_{i}}(x_{2},x_{3}) 1(x_{1} \in S^{(n)}_{i}), \quad \vc{x} = (x_{1},x_{2},x_{3}) \in \dd{R}_{+}^{3},
\end{align}
and define its auxiliary transform functions $\psi^{(n)}_{i}$ for $i = 1,2,\ldots,K$ as
\begin{align}
\label{eq:phi-theta-r}
 & \psi^{(n)}_{i}(\theta_{i}) = \dd{E}\left[e^{n^{-1/2}\theta_{i} L^{(n)}}g^{(n)}_{n^{-1/2} \theta_{i}}(R^{(n)}) 1(L^{(n)} \in S^{(n)}_{i}) \right], \qquad \vc{\theta} \in \dd{R}^{K-1} \times \dd{R}_{-}.
\end{align}
Then, by the assumption (\sect{preliminaries}.a) in \sectn{HT-a} and \lem{MLQ-gn-limit}, we have
\begin{align}
\label{eq:MLQ-finite1}
  \psi^{(n)}_{i}(\theta_{i}) < \infty, \quad \dd{E}\left[f^{(n)}_{n^{-1/2}\vc{\theta}}(L^{(n)},R^{(n)})\right] < \infty, \qquad \mbox{uniformly in $n \ge 1$},
\end{align}
because $n^{-1/2} \theta_{i} L^{(n)}1(L^{(n)} \le \ell^{(n)}_{i}) \le \theta_{i} (\ell_{i} + o(n^{-1/2}))$ for $\theta_{i} > 0$ for $i=1,2,\ldots,K-1$. Furthermore, by \lem{MLQ-gn-limit}, we have

\begin{corollary}\rm
\label{cor:MLQ-gn-limit}
For $\vc{\theta}$ in any bounded subset of $\dd{R}^{K-1} \times \dd{R}_{-}$,
\begin{align}
\label{eq:MLQ-psi-n-limit}
 & \lim_{n \to \infty} |\psi^{(n)}_{i}(\theta_{i}) - \varphi^{(n)}_{i}(\theta_{i}) d^{(n)}_{i}| = 0, \qquad  i \in J_{K}.
\end{align}
\end{corollary}

Those results guarantees our computations below for $\vc{\theta} \in \dd{R}^{K-1} \times \dd{R}_{-}$. Based on them, we compute $\dd{E}\left[ \sr{H}^{(n)}f(X^{(n)}(0))\right]$ in \eq{BAR1} for $f^{(n)}_{n^{-1/2}\vc{\theta}}$. Then, 
\begin{align}
\label{eq:MLQ-H1}
   \dd{E}\left[ \sr{H}^{(n)}f^{(n)}_{n^{-1/2} \vc{\theta}}(X^{(n)}(0))\right] & = \sum_{i = 1}^{K} \left(\lambda^{(n)}_{i} \eta^{(n)}(n^{-1/2}\theta_{i}) + \mu^{(n)}_{i} \zeta^{(n)} (n^{-1/2}\theta_{i})\right) \psi^{(n)}_{i}(\theta_{i}) \nonumber\\
 & \quad - \mu^{(n)}_{1} \zeta^{(n)}(n^{-1/2}\theta_{1}) \dd{E}\left[1(L^{(n)}=0) g^{(n)}_{n^{-1/2} \theta_{1}}(R^{(n)}) )\right].
\end{align}
Since $\lambda^{(n)}_{i} - \mu^{(n)}_{i} = b_{i} n^{-1/2} + o(n^{-1/2})$, it follows from \lem{MLQ-eta-zeta1}, $\beta_{i} = 2b_{i}/\sigma_{i}^{2}$ and $\lambda_{i} = \mu_{i}$ that
\begin{align*}
 \lambda^{(n)}_{i} \eta^{(n)}(n^{-1/2}\theta_{i}) + \mu^{(n)}_{i} \zeta^{(n)}(n^{-1/2}\theta_{i}) & = b_{i} \theta_{i} n^{-1} + \frac 12 (\lambda_{i} \sigma_{A}^{2} + \mu_{i} \sigma_{S}^{2}) n^{-1} \theta_{i}^{2} + |\theta_{i}| o(n^{-1})\\
  & = \frac 12 \sigma_{i}^{2} (\beta_{i} + \theta_{i}) \theta_{i} n^{-1} + |\theta_{i}| o(n^{-1}).
\end{align*}
Hence, from \eq{MLQ-H1}, we have
\begin{align}
\label{eq:MLQ-H2}
 & \dd{E}\left[ \sr{H}^{(n)}f^{(n)}_{n^{-1/2} \vc{\theta}}(X^{(n)}(0))\right] = \frac 12 \sum_{i = 1}^{K} \sigma_{i}^{2} \left(\beta_{i} + \theta_{i} \right) \theta_{i} n^{-1} \psi^{(n)}_{i}(\theta_{i}) \nonumber\\
 & \quad + \left(1 - \frac 12 \sigma_{S}^{2} n^{-1/2} \theta_{1}\right) \mu^{(n)}_{1} n^{-1/2} \theta_{1} \dd{E}\left[1(L^{(n)}=0) g^{(n)}_{n^{-1/2} \theta_{1}}(R^{(n)})\right] + |\vc{\theta}| o(n^{-1}),
\end{align}
where $|\vc{\theta}| = \sum_{i=1}^{K} |\theta_{i}|$.

We next compute the Palm expectation terms in \eq{BAR1}. In what follows, we use the following simplified notations.
\begin{align*}
 & \widehat{T}^{(n)}_{A} \equiv T_{A} \wedge n^{1/2}, \qquad \widehat{T}^{(n)}_{s} \equiv T_{S} \wedge n^{1/2},\\
 & \widehat{R}^{(n)}_{e} \equiv R^{(n)}_{e}(0) \wedge n^{1/2}, \qquad \widehat{R}^{(n)}_{d} \equiv R^{(n)}_{d}(0-) \wedge n^{1/2}.
\end{align*}

Since $g^{(n)}_{\theta_{i}}(R^{(n)}(0)) \not= g^{(n)}_{\theta_{i}}(R^{(n)}(0-))$ under the Palm distributions only when $L^{(n)}(0-) = \ell^{(n)}_{i}$ under $\dd{P}^{(n)}_{e}$ or $L^{(n)}(0) = \ell^{(n)}_{i}$ under $\dd{P}^{(n)}_{d}$, we have, from the definitions of $\Delta_{e}$ and $\Delta_{e}$,
\begin{align*}
 & \dd{E}^{(n)}_{e}\left[ \Delta_{e} f^{(n)}_{n^{-1/2}\vc{\theta}}(X^{(n)})(0) 1(\cap_{i=1}^{K-1} \{L^{(n)}(0-) \not= \ell^{(n)}_{i}\})\right] = 0,\\
 & \dd{E}^{(n)}_{d}\left[\Delta_{d} f^{(n)}_{n^{-1/2}\vc{\theta}}(X^{(n)})(0) 1(\cap_{i=1}^{K-1} \{L^{(n)}(0) \not= \ell^{(n)}_{i}\})\right] = 0.
\end{align*}
Hence, define
\begin{align*}
 & E^{(n)}_{\Delta}(\vc{\theta}) = \dd{E}^{(n)}_{e}\left[\Delta_{e} f^{(n)}_{n^{-1/2}\vc{\theta}}(X^{(n)})(0) \right] + \dd{E}^{(n)}_{d}\left[\Delta_{d} f^{(n)}_{n^{-1/2}\vc{\theta}}(X^{(n)})(0) \right],\\
 & E^{(n)}_{\Delta,e,i}(\vc{\theta}) = \dd{E}^{(n)}_{e}\left[\Delta_{e} f^{(n)}_{n^{-1/2}\vc{\theta}}(X^{(n)})(0) 1(L^{(n)}(0-) = \ell^{(n)}_{i}) \right], \qquad i = 1,2,\ldots,K-1,\\
 & E^{(n)}_{\Delta,d,i}(\vc{\theta}) = \dd{E}^{(n)}_{d}\left[\Delta_{d} f^{(n)}_{n^{-1/2}\vc{\theta}}(X^{(n)})(0) 1(L^{(n)}(0) = \ell^{(n)}_{i}) \right], \qquad i = 1,2,\ldots,K-1,
\end{align*}
then
\begin{align}
\label{eq:MLQ-df}
  E^{(n)}_{\Delta}(\vc{\theta}) = \sum_{i = 1}^{K-1} \left( E^{(n)}_{\Delta,e,i}(\vc{\theta}) + E^{(n)}_{\Delta,d,i}(\vc{\theta})\right),
\end{align}

Recall the definitions of $X^{(n)}_{1}(0)$ and $\Delta_{e} f^{(n)}_{n^{-1/2}\vc{\theta}}(X^{(n)})(0)$, and compute $E^{(n)}_{\Delta,e,i}(\vc{\theta})$:
\begin{align}
\label{eq:MLQ-df1}
 & E^{(n)}_{\Delta,e,i}(\vc{\theta}) = \dd{E}^{(n)}_{e}[(f^{(n)}_{n^{-1/2}\vc{\theta}}(X^{(n)}_{1}(0)) - f^{(n)}_{n^{-1/2}\vc{\theta}}(X^{(n)}(0-))) 1(L^{(n)}(0-) = \ell^{(n)}_{i})] \nonumber\\
  & \quad = \dd{E}^{(n)}_{e} \Big[e^{n^{-1/2}\theta_{i+1} (\ell^{(n)}_{i} + 1)} g^{(n)}_{n^{-1/2}\theta_{i+1}}(R^{(n)}_{1}(0))1(L^{(n)}_{1}(0) = \ell^{(n)}_{i}+1)  \nonumber\\
  & \hspace{15ex} - e^{n^{-1/2}\theta_{i} \ell^{(n)}_{i}} g^{(n)}_{n^{-1/2}\theta_{i}}(R^{(n)}(0-)) 1(L^{(n)}(0-) = \ell^{(n)}_{i}) \Big] \nonumber\\
 & \quad = e^{n^{-1/2}\theta_{i+1} \ell^{(n)}_{i}} \dd{E}^{(n)}_{e}[e^{-\zeta^{(n)}(n^{-1/2}\theta_{i+1}) \widehat{R}^{(n)}_{d}} 1(L^{(n)}(0-) = \ell^{(n)}_{i})]  \nonumber\\
  & \qquad - e^{n^{-1/2}\theta_{i} \ell^{(n)}_{i}} \dd{E}^{(n)}_{e}[e^{-\zeta^{(n)}(n^{-1/2}\theta_{i}) \widehat{R}^{(n)}_{d}} 1(L^{(n)}(0-) = \ell^{(n)}_{i})],
\end{align}
where the last equality is obtained by applying \eq{MLQ-boundary1}. Similarly, we have
\begin{align}
\label{eq:MLQ-df2}
 & E^{(n)}_{\Delta,d,i}(\vc{\theta})  = \dd{E}^{(n)}_{d}[(f^{(n)}_{n^{-1/2}\vc{\theta}}(X^{(n)}(0)) - f^{(n)}_{n^{-1/2}\vc{\theta}}(X^{(n)}_{1}(0))) 1(L^{(n)}(0) = \ell^{(n)}_{i})]\nonumber\\
  & \quad = \dd{E}^{(n)}_{d} \Big[e^{n^{-1/2}\theta_{i} \ell^{(n)}_{i}} g^{(n)}_{n^{-1/2}\theta_{i}}(R^{(n)}(0))1(L^{(n)}(0) = \ell^{(n)}_{i})  \nonumber\\
  & \hspace{15ex} - e^{n^{-1/2}\theta_{i+1} (\ell^{(n)}_{i}+1)} g^{(n)}_{n^{-1/2}\theta_{i+1}}(R^{(n)}_{1}(0)) 1(L^{(n)}_{1}(0-) = \ell^{(n)}_{i}+1) \Big] \nonumber\\
 & \quad = e^{n^{-1/2}\theta_{i} (\ell^{(n)}_{i}+1)} \dd{E}^{(n)}_{d}[e^{-\eta^{(n)}(n^{-1/2}\theta_{i}) \widehat{R}^{(n)}_{e}} 1(L^{(n)}(0) = \ell^{(n)}_{i})]  \nonumber\\
  & \qquad - e^{n^{-1/2}\theta_{i+1} (\ell^{(n)}_{i}+1)} \dd{E}^{(n)}_{d}[e^{-\eta^{(n)}(n^{-1/2}\theta_{i+1}) \widehat{R}^{(n)}_{e}} 1(L^{(n)}(0) = \ell^{(n)}_{i})].
\end{align}

Then, by \eq{MLQ-df}, \eq{MLQ-df1} and \eq{MLQ-df2},
\begin{align}
\label{eq:MLQ-En1}
   E^{(n)}_{\Delta}(\vc{\theta}) & = \sum_{i = 1}^{K-1} \Big(e^{n^{-1/2}\theta_{i+1} \ell^{(n)}_{i}} \dd{E}^{(n)}_{e}[e^{-\zeta^{(n)}(n^{-1/2}\theta_{i+1}) \widehat{R}^{(n)}_{d}} 1(L^{(n)}(0-) = \ell^{(n)}_{i})]  \nonumber\\
  & \hspace{10ex} - e^{n^{-1/2}\theta_{i} \ell^{(n)}_{i}} \dd{E}^{(n)}_{e}[e^{-\zeta^{(n)}(n^{-1/2}\theta_{i}) \widehat{R}^{(n)}_{d}} 1(L^{(n)}(0-) = \ell^{(n)}_{i})] \Big) \nonumber\\
  & \quad + \sum_{i = 1}^{K-1} \Big(e^{n^{-1/2}\theta_{i} (\ell^{(n)}_{i}+1)} \dd{E}^{(n)}_{d}[e^{-\eta^{(n)}(n^{-1/2}\theta_{i}) \widehat{R}^{(n)}_{e}} 1(L^{(n)}(0) = \ell^{(n)}_{i})]  \nonumber\\
  & \hspace{10ex} - e^{n^{-1/2}\theta_{i+1} (\ell^{(n)}_{i}+1)} \dd{E}^{(n)}_{d}[e^{-\eta^{(n)}(n^{-1/2}\theta_{i+1}) \widehat{R}^{(n)}_{e}} 1(L^{(n)}(0) = \ell^{(n)}_{i})]\Big).
\end{align}

Thus, applying \eq{MLQ-H2}, \eq{MLQ-df1} and \eq{MLQ-df2} to \eq{BAR1}, we have a pre-limit BAR for the $K$-level $GI/G/1$ queue.

\begin{lemma}
\label{lem:MLQ-BAR1}
For each fixed $\vc{\theta} \in \dd{R}^{K-1} \times \dd{R}_{-}$, we have, as $n \to \infty$,
\begin{align}
\label{eq:MLQ-BAR1}
 & \frac 12 \sum_{i = 1}^{K} \sigma_{i}^{2} \big((\beta_{i} + \theta_{i}) \theta_{i} n^{-1} + o(n^{-1})\big) \psi^{(n)}_{i}(\theta_{i}) \nonumber\\
 & \quad + \left(\mu^{(n)}_{1} n^{-1/2}\theta_{1} \left(1 - \frac 12 \sigma_{S}^{2} n^{-1/2} \theta_{1}\right) + o(n^{-1})\right)  \dd{E}[ 1(L^{(n)} = 0) g^{(n)}_{ n^{-1/2} \theta_{1}}(R^{(n)})]  \nonumber\\
 & \quad + \alpha^{(n)}_{e} E^{(n)}_{\Delta}(\vc{\theta}) = 0.
\end{align}
\end{lemma}

We call BAR \eq{MLQ-BAR1} an asymptotic BAR, which is a starting point of our proof of \thr{MLQ-main}. We will divide this BAR by $n^{-1}$, and take its limit as $n \to \infty$. To perform this expansion, we need to compute $n^{-1} \alpha^{(n)}_{e} E^{(n)}_{\Delta}(\vc{\theta})$ as $n \to \infty$. For this computation, we consider asymptotic properties of $R^{(n)}_{e}$ and $R^{(n)}_{d}$ in distribution as $n \to \infty$.

\begin{lemma}
\label{lem:MLQ-moment}
For $y \in \dd{R}_{+}$,
\begin{align}
\label{eq:MLQ-Re 1}
 & \dd{P}[R^{(n)}_{e} > y] \le (\min_{i \in J_{K}} \lambda^{(n)}_{i})^{-1} \alpha^{(n)}_{e} \dd{E}[(T_{A}-y) 1(T_{A} > y)],\\
\label{eq:MLQ-Rd 1}
 & \dd{P}[R^{(n)}_{d} > y] \le (\min_{i \in J_{K}} \mu^{(n)}_{i})^{-1} \left( \mu^{(n)}_{1} \dd{P}[T_{S} > y] + \alpha^{(n)}_{e} \dd{E}[(T_{S}-y) 1(T_{S} > y)]\right),
\end{align}
and therefore, for $k =1,2$ and $n \ge 1$,
\begin{align}
\label{eq:MLQ-Re 2}
 & (k+1) \dd{E}[(R^{(n)}_{e})^{k}] \le (\min_{i \in J_{K}} \lambda^{(n)}_{i})^{-1} \alpha^{(n)}_{e} \dd{E}[(T_{A})^{k+1}], \\
\label{eq:MLQ-Rd 2}
 & (k+1) \dd{E}[(R^{(n)}_{d})^{k}] \le (k+1) \dd{E}[T_{S}^{k}] + (\min_{i \in J_{K}} \mu^{(n)}_{i})^{-1} \alpha^{(n)}_{e} \dd{E}[T_{S}^{k+1}].
\end{align}
In particular, $\{(R^{(n)}_{e})^{k}; n \ge 1\}$ and $\{(R^{(n)}_{d})^{k} 1(L^{(n)}(0) \ge 1); n \ge 1\}$ are uniformly integrable for $k=1,2$ by \eq{MLQ-Re 2} and \eq{MLQ-Rd 2}.
\end{lemma}
This lemma easily follows from \eq{BAR1} by appropriately chosen test functions. For example, we choose $f(\vc{x}) = (x_{2} - y) 1(x_{2} \ge y)$, which satisfies the assumption \eq{test-f}. Applying this $f$ to \eq{BAR1}, we have
\begin{align*}
  \sum_{i \in J_{K}} \lambda^{(n)}_{i} \dd{P}[R^{(n)}_{e} > y, L^{(n)} \in S_{i}] = \alpha_{e} \dd{E}[(T_{A}-y) 1(T_{A} > y)],
\end{align*}
which yields \eq{MLQ-Re 1}. Similarly, \eq{MLQ-Rd 1} is proved. The following lemma is essentially the same as Lemma 4.4 of \cite{Miya2025}. 

\begin{lemma}
\label{lem:MLQ-t-moment}
For $k=1,2$, as $n \to \infty$,
\begin{align}
\label{eq:MLQ-Tu 1}
 & |\dd{E}[\widehat{T}_{u}^{k}] - \dd{E}[T_{u}^{k}]| = \dd{E}[T_{u}^{k} 1(T_{u} > n^{1/2})] = o(n^{-(3-k)/2}), \quad u = A,S,\\
\label{eq:MLQ-Rv 1}
 & |\dd{E}[(\widehat{R}^{(n)}_{v})^{k}] - \dd{E}[(R^{(n)}_{v})^{k}]| = \dd{E}[(R^{(n)}_{v})^{k} 1(R^{(n)}_{v} > n^{1/2})] = o(n^{-(2-k)/2}), \quad v = e, d.
\end{align}
\end{lemma}

We prepare one more lemma and its corollary, which are helpful for the expansion of $n^{-1} \alpha^{(n)}_{e} E^{(n)}_{\Delta}(\vc{\theta})$.

\begin{lemma}
\label{lem:MLQ-Red}
For $k=1,2$ and $n \ge 1$,
\begin{align}
\label{eq:MLQ-Rd-1}
 & \mu^{(n)}_{1} k \dd{E}[(R^{(n)}_{d})^{k-1} (1 - 1(R^{(n)}_{d} \ge n^{1/2})) 1(0 < L^{(n)} \le \ell^{(n)}_{1})] \nonumber\\
 & \quad = \alpha^{(n)}_{e} \Big(-\dd{E}^{(n)}_{e}[(\widehat{R}^{(n)}_{d})^{k} 1(L^{(n)}(0-) = \ell^{(n)}_{1})] + \dd{E}[\widehat{T}_{S}^{k}] \dd{P}^{(n)}_{d}[L^{(n)}(0) \in S^{(n)}_{1}]\Big), \\
\label{eq:MLQ-Rd-j}
 & \mu^{(n)}_{j} k \dd{E}[(R^{(n)}_{d})^{k-1} (1 - 1(R^{(n)}_{d} \ge n^{1/2})) 1(L^{(n)} \in S^{(n)}_{j})] \nonumber\\
 & \quad = \alpha^{(n)}_{e} \dd{E}^{(n)}_{e}[(\widehat{R}^{(n)}_{d})^{k} (1(L^{(n)}(0-) = \ell^{(n)}_{j-1}) - 1(L^{(n)}(0-) = \ell^{(n)}_{j}))] \nonumber\\
 & \qquad + \alpha^{(n)}_{e} \dd{E}[\widehat{T}_{S}^{k}] \dd{P}^{(n)}_{d}[L^{(n)}(0) \in S^{(n)}_{j}], \qquad j = 2,3, \ldots, K-1,\\
\label{eq:MLQ-Rd-K}
 & \mu^{(n)}_{K} k \dd{E}[(R^{(n)}_{d})^{k-1} (1 - 1(R^{(n)}_{d} \ge n^{1/2})) 1( L^{(n)} \in S^{(n)}_{K})]\nonumber\\
 & \quad = \alpha^{(n)}_{e} \Big( \dd{E}^{(n)}_{e}[(\widehat{R}^{(n)}_{d})^{k} 1(L^{(n)}(0-) = \ell^{(n)}_{K-1})] + \dd{E}[\widehat{T}_{S}^{k}] \dd{P}^{(n)}_{d}[L^{(n)}(0) \in S^{(n)}_{K}]\Big),
\end{align}
Similarly, for $R^{(n)}_{e}(\cdot)$,
\begin{align}
\label{eq:MLQ-Re-1}
 & \lambda^{(n)}_{1} k \dd{E}[(R^{(n)}_{e})^{k-1} (1 - 1(R^{(n)}_{e} \ge n^{1/2})) 1(L^{(n)} \in S^{(n)}_{1})]  \nonumber\\
 & \quad = \alpha^{(n)}_{e}\Big( \dd{E}_{d}[(\widehat{R}^{(n)}_{e})^{k} 1(L^{(n)}(0) = \ell^{(n)}_{1})] + \dd{E}[\widehat{T}_{A}^{k}] \dd{P}^{(n)}_{e}[L^{(n)}(0) \in S^{(n)}_{1}]\Big),\\
\label{eq:MLQ-Re-j}
 & \lambda^{(n)}_{j} k \dd{E}[(R^{(n)}_{e})^{k-1} (1 - 1(R^{(n)}_{e} \ge n^{1/2})) 1(L^{(n)} \in S^{(n)}_{j})]  \nonumber\\
 & \quad = \alpha^{(n)}_{e} \dd{E}_{d}[((\widehat{R}^{(n)}_{e})^{k} 1(L^{(n)}(0) = \ell^{(n)}_{j}) - 1(L^{(n)}(0) = \ell^{(n)}_{j-1}))]  \nonumber\\
 & \qquad + \alpha^{(n)}_{e} \dd{E}[\widehat{T}_{A}^{k}] \dd{P}^{(n)}_{e}[L^{(n)}(0) \in S^{(n)}_{j}], \qquad j = 2,3,\ldots,K-1,\\
\label{eq:MLQ-Re-K}
 & \lambda^{(n)}_{K} k \dd{E}[(R^{(n)}_{e})^{k-1} (1 - 1(R^{(n)}_{e} \ge n^{1/2})) 1(L^{(n)} \in S^{(n)}_{K})] \nonumber\\
 & \quad = \alpha^{(n)}_{e}\Big(- \dd{E}_{d}[(\widehat{R}^{(n)}_{e})^{k} 1(L^{(n)}(0) = \ell^{(n)}_{K-1})] + \dd{E}[\widehat{T}_{A}^{k}] \dd{P}^{(n)}_{e}[L^{(n)}(0) \in S^{(n)}_{K}] \Big).
\end{align}
\end{lemma}
\begin{proof}
For $k \ge 1$, we first prove \eq{MLQ-Rd-1}. We apply $f(\vc{x}) = (x_{3} \wedge n^{1/2})^{k} 1(x_{1} \le \ell^{(n)}_{1})$ to \eq{BAR1}, then we have \eq{MLQ-Rd-1} because $\dd{E}^{(n)}_{d}[(\widehat{R}^{(n)}_{d}(0))^{k}] = \dd{E}[\widehat{T}_{S}^{k}]$ and
\begin{align*}
 & \sr{H}^{(n)} f(X^{(n)}(t)) = - k \mu^{(n)}_{1} (R^{(n)}_{d}(t))^{k-1} 1(R^{(n)}_{d} < n^{1/2}, 1 \le L^{(n)}(t) \le \ell^{(n)}_{1}),\\
 & \Delta_{e} f(X^{(n)})(t) = - (\widehat{R}^{(n)}_{d}(t-))^{k} 1(L^{(n)}(t-) = \ell^{(n)}_{1}) 1(\Delta N^{(n)}_{e}(t) = 1),\\
 & \Delta_{d} f(X^{(n)})(t) = (\widehat{R}^{(n)}_{d}(t))^{k} 1(L^{(n)}(t) \in S^{(n)}_{1}) 1(\Delta N^{(n)}_{d}(t) = 1).
\end{align*}
Similarly, \eq{MLQ-Rd-j}, \eq{MLQ-Rd-K} are obtained from \eq{BAR1} by applying $f(\vc{x}) = (x_{3} \wedge n^{1/2})^{k} 1(x_{1} \in S^{(n)}_{j})$, $f(\vc{x}) = (x_{3} \wedge n^{1/2})^{k} 1(x_{1} \in S^{(n)}_{K})$, respectively. \eq{MLQ-Re-1}, \eq{MLQ-Re-j} and \eq{MLQ-Re-K} are similarly proved by applying $f(\vc{x}) = (x_{2} \wedge n^{1/2})^{k} 1(x_{1} \in S^{(n)}_{1})$, $f(\vc{x}) = (x_{2} \wedge n^{1/2})^{k} 1(x_{1} \in S^{(n)}_{j})$ and $f(\vc{x}) = (x_{2} \wedge n^{1/2})^{k} 1(x_{1} \in S^{(n)}_{K})$, respectively.
\end{proof}

The following corollary is immediate from Lemmas \lemt{MLQ-moment} and \lemt{MLQ-Red}.

\begin{corollary}
\label{cor:MLQ-Re-Rs 1}
For each $k=1,2$ and $i \in J_{K-1}$, $\dd{E}_{e}[(\widehat{R}^{(n)}_{d})^{k} 1(L^{(n)}(0-) = \ell^{(n)}_{i})]$ and $\dd{E}_{d}[(\widehat{R}^{(n)}_{e})^{k} 1(L^{(n)}(0) = \ell^{(n)}_{i})]$ are uniformly bounded for $n \ge 1$.
\end{corollary}

Note that the limits of $\psi^{(n)}_{i}(\theta_{i})$ for $i \in J_{K}$ should be obtained from the BAR \eq{BAR1}. However, we will see that it requires the relation among the limits of $\varphi^{(n)}_{i}(0)$, which is not easy to get from \eq{BAR1}. To consider this problem, we introduce the following quantities. Define, for $i \in J_{K-1}$,
\begin{align}
\label{eq:MLQ-Dn-i}
 \Delta^{(n)}_{i} = \alpha^{(n)}_{e} \Big( & \dd{E}^{(n)}_{e}[\widehat{R}^{(n)}_{d} 1(L^{(n)}(0-) = \ell^{(n)}_{i})] + \dd{E}^{(n)}_{d}[\widehat{R}^{(n)}_{e} 1(L^{(n)}(0) = \ell^{(n)}_{i})] \nonumber\\
 & - \dd{P}^{(n)}_{e}[L^{(n)}(0-) = \ell^{(n)}_{i}] \Big).
\end{align}
Then, $\Delta^{(n)}_{i}$ is well defined and finite by \cor{MLQ-Re-Rs 1}, and we have

\begin{lemma}
\label{lem:MLQ-R2-1}
\begin{align}
\label{eq:MLQ-Dn-j+}
 & n^{-1/2} \sum_{i=1}^{j} b_{i} \dd{P}[L^{(n)} \in S^{(n)}_{i}]  + \mu^{(n)}_{1} \dd{P}[L^{(n)} = 0] = \Delta^{(n)}_{j} + o(n^{-1/2}), \qquad j \in J_{K-1},\\
\label{eq:MLQ-Pn-0}
 & \mu^{(n)}_{1} \dd{P}[L^{(n)} = 0] = - n^{-1/2} \sum_{i=1}^{K} b_{i} \dd{P}[L^{(n)} \in S^{(n)}_{i}] +  o(n^{-1/2}),\\
\label{eq:MLQ-Dn-j-}
 & \Delta^{(n)}_{j} = - n^{-1/2} \sum_{i=j+1}^{K} b_{i} \dd{P}[L^{(n)} \in S^{(n)}_{i}] + o(n^{-1/2}), \qquad j \in J_{K-1}.
\end{align}
\end{lemma}

\begin{proof}
Subtract both sides of \eq{MLQ-Rd-1} from those of \eq{MLQ-Re-1} for $k=1$, then we have, as $n \to \infty$, 
\begin{align}
\label{eq:MLQ-Rn-1}
  & (\lambda^{(n)}_{1} - \mu^{(n)}_{1}) \dd{P}[L^{(n)} \in S^{(n)}_{1}]  + \mu^{(n)}_{1} \dd{P}[L^{(n)} = 0] + o(n^{-1/2}) \nonumber\\
  & \quad = \alpha^{(n)}_{e} \Big(\dd{E}^{(n)}_{e}[\widehat{R}^{(n)}_{d} 1(L^{(n)}(0-) = \ell^{(n)}_{1})] + \dd{E}^{(n)}_{d}[\widehat{R}^{(n)}_{e} 1(L^{(n)}(0) = \ell^{(n)}_{1})] \nonumber\\
 & \hspace{11ex} + \dd{P}^{(n)}_{e}[L^{(n)}(0) \le \ell^{(n)}_{1}] - \dd{P}^{(n)}_{d}[L^{(n)}(0) \le \ell^{(n)}_{1}] \Big).
\end{align}
This yields \eq{MLQ-Dn-j+} for $j=1$ because $\lambda^{(n)}_{1} - \mu^{(n)}_{1} = n^{-1/2} b_{1} + o(n^{-1/2})$ and $\dd{P}^{(n)}_{e}[L^{(n)}(0) \le \ell^{(n)}_{1}] - \dd{P}^{(n)}_{d}[L^{(n)}(0) \le \ell^{(n)}_{1}] = - \dd{P}^{(n)}_{e}[L^{(n)}(0-) = \ell^{(n)}_{1}]$ by \lem{MLQ-basic 1}. To show \eq{MLQ-Dn-j+} for $j=2,3,\ldots,K-1$, we first replace this $j$ by $j'$, and let $k=1$. Then, subtract both sides of \eq{MLQ-Re-j} from those of \eq{MLQ-Rd-j} for each $j=2,3,\ldots,j' \le K-1$, then sum up them from $j=2$ to $j=j'$ together with both sides of \eq{MLQ-Rn-1}. This yields \eq{MLQ-Dn-j+} with $j=j'$ for $j'=2,3,\ldots,K-1$. Similarly, subtracting both sides of \eq{MLQ-Rd-K} from those of \eq{MLQ-Re-K} for $k=1$, we have
\begin{align}
\label{eq:MLQ-Rn-K}
 & n^{-1/2} b_{K} \dd{P}[L^{(n)} \in S^{(n)}_{K}] = - \Delta^{(n)}_{K-1} + o(n^{-1/2}).
\end{align}
Substituting this $\Delta^{(n)}_{K-1}$ into \eq{MLQ-Dn-j+} for $j=K-1$, we have \eq{MLQ-Pn-0}. Finally, \eq{MLQ-Dn-j-} is obtained by substituting \eq{MLQ-Pn-0} into \eq{MLQ-Dn-j+}. Note that \eq{MLQ-Rn-K} is identical with \eq{MLQ-Dn-j-} for $j=K-1$.
\end{proof}

Our next task is to expand the jump component $E^{(n)}_{\Delta}(\vc{\theta})$ of \eq{MLQ-En1} using $\Delta^{(n)}_{i}$ for large $n$. For this, we will use the following Taylor expansions.

\begin{lemma}
\label{lem:MLQ-asymp 1}
For each fixed $\theta_{i}$ and $\theta_{i+1}$, as $n \to \infty$,
\begin{align}
\label{eq:MLQ-exp 1}
 & e^{n^{-1/2} \theta_{i} \ell^{(n)}_{i}} = e^{\theta_{i} \ell_{i}} (1 + o(\theta_{i} n^{-1/2})), \\
 & e^{n^{-1/2}\theta_{i} (\ell^{(n)}_{i}+1)} = e^{\theta_{i} \ell_{i}} (1 + \theta_{i}n^{-1/2} + o(\theta_{i} n^{-1/2})), \quad i=1,2,\\
\label{eq:MLQ-exp 2}
 & e^{n^{-1/2}\theta_{i}\ell^{(n)}_{i}} - e^{n^{-1/2}\theta_{i+1} \ell^{(n)}_{i}} = e^{\theta_{i} \ell_{i}} - e^{\theta_{i+1} \ell_{i}} + e^{(\theta_{i} \vee \theta_{i+1}) \ell_{i}} ( o(\theta_{i} n^{-1/2}) + o(\theta_{i+1} n^{-1/2})),\\
\label{eq:MLQ-exp 3}
 & e^{n^{-1/2}\theta_{i}(\ell^{(n)}_{i}+1)} - e^{n^{-1/2}\theta_{i+1} (\ell^{(n)}_{i}+1)}  \nonumber\\
 & \quad = e^{\theta_{i}\ell_{i}} - e^{\theta_{i+1} \ell_{i}}  \nonumber\\
 & \qquad + n^{-1/2} (e^{\theta_{i} \ell_{i}} \theta_{i} - e^{\theta_{i+1} \ell_{i}} \theta_{i+1}) + e^{(\theta_{i} \vee \theta_{i+1}) \ell_{i}} ( o(\theta_{i} n^{-1/2}) + o(\theta_{i+1} n^{-1/2})).
\end{align}
\end{lemma}

\begin{lemma}
\label{lem:MLQ-asymp 2}
Let $P_{n}(x) = \sum_{k=1}^{n} x^{k}$ for $x \in \dd{R}$ and integer $n \ge 1$, then, for each fixed $\theta_{i}$, as $n \to \infty$, for $i=1,2$,
\begin{align}
\label{eq:MLQ-exp-eta 1}
 e^{-\eta^{(n)}_{i}(n^{-1/2}\theta_{i}) \widehat{R}^{(n)}_{e}} & = 1 - n^{-1/2}\theta_{i} \widehat{R}^{(n)}_{e} \nonumber\\
 & \quad - 2^{-1} (n^{-1/2}\theta_{i})^{2} [ \sigma_{A}^{2} \widehat{R}^{(n)}_{e} - (\widehat{R}^{(n)}_{e})^{2} ] + P_{3}(\widehat{R}^{(n)}_{e}) o((n^{-1/2}\theta_{i})^{2}),\\
\label{eq:MLQ-exp-zeta 1}
 e^{-\zeta^{(n)}(n^{-1/2}\theta_{i}) \widehat{R}^{(n)}_{d}} & = 1 + n^{-1/2}\theta_{i} \widehat{R}^{(n)}_{d} \nonumber\\
 & \quad - 2^{-1} (n^{-1/2}\theta_{i})^{2} [ \sigma_{S}^{2} \widehat{R}^{(n)}_{d} - (\widehat{R}^{(n)}_{d})^{2} ] + P_{3}(\widehat{R}^{(n)}_{d}) o((n^{-1/2}\theta_{i})^{2}).
\end{align}
\end{lemma}

These lemmas are essentially the same as Lemmas 4.7 and 4.8 for $r = n^{-1/2}$ of \cite{Miya2025}. For $j \in J_{K}$, let $\vcn{e}_{j}$ be the vector in $\dd{R}^{K}$ whose $j$-th entry is unit and all the other entries vanish, then, applying \eq{MLQ-exp-eta 1} and \eq{MLQ-exp-zeta 1} to \eq{MLQ-En1} and recalling the definition \eq{MLQ-Dn-i} of $\Delta^{(n)}_{i}$, we have
\begin{lemma}
\label{lem:MLQ-En-j}
\begin{align}
\label{eq:MLQ-En-j}
  \alpha^{(n)}_{e} E^{(n)}_{\Delta}(\theta_{j}\vcn{e}_{j}) = n^{-1/2} \theta_{j} (e^{\ell_{j-1} \theta_{j}} \Delta^{(n)}_{j-1} - e^{\ell_{j} \theta_{j}} \Delta^{(n)}_{j}) + o(n^{-1}), \qquad j \in J_{K},
\end{align}
where $\Delta^{(n)}_{0} = \Delta^{(n)}_{K} = 0$, and therefore $\alpha^{(n)}_{e} E^{(n)}_{\Delta}(\theta_{j}\vcn{e}) = O(n^{-1})$.
\end{lemma}

\subsection{Computing the limiting distribution of $\nu^{(n)}$}
\label{sec:computing}

We compute the weak limit of $\nu^{(n)}$ in two steps. In the first step, we prove that $d^{(n)}_{j}$ converges to $d_{j}$ as $n \to \infty$ for $j \in J_{K}$. This implies the tightness of $\{\nu^{(n)}; n \ge 1\}$. Using this fact, we compute the limit of $\nu^{(n)}$ through its moment generating function in the 2nd step.

\noindent {(\bf 1st Step)} $\;$ Recalling that $d^{(n)}_{i} = \dd{P}[L^{(n)} \in S^{(n)}_{i}]$, we aim to compute the limit of $d^{(n)}_{i}$ as $n \to \infty$. Let
\begin{align*}
  B^{(n)}_{j} = \sum_{i=j}^{K} b_{i} d^{(n)}_{i}, \qquad j \in J_{K},
\end{align*}
then we have the following lemma from Lemmas \lemt{MLQ-BAR1} and \lemt{MLQ-En-j}.
\begin{lemma}
\label{lem:MLQ-BAR2-i}
\begin{align}
\label{eq:MLQ-BAR2-i}
 \frac 12 \sigma_{i}^{2} (\beta_{i} + \theta_{i}) \psi^{(n)}_{i}(\theta_{i}) & = e^{\ell_{i-1} \theta_{i}} B^{(n)}_{i} - e^{\ell_{i} \theta_{i}} B^{(n)}_{i+1} + o(1), \qquad i \in J_{K},
\end{align}
where $B^{(n)}_{K+1} = 0$, and recall that $\ell_{0} = 0$.
\end{lemma}
\begin{proof}
From \eq{MLQ-BAR1}, \eq{MLQ-Pn-0}, \eq{MLQ-Dn-j-} and \lem{MLQ-En-j}, we have
\begin{align*}
 \frac 12 \sigma_{1}^{2} (\beta_{1} + \theta_{1}) \psi^{(n)}_{1}(\theta_{1}) & = - \mu^{(n)}_{n} n^{1/2} \dd{P}[L^{(n)} = 0] - n \alpha^{(n)}_{e} \theta_{1}^{-1} E^{(n)}_{\Delta}(\theta_{1}\vcn{e}_{1}) + o(1)\nonumber\\
 & = B^{(n)}_{1} - e^{\ell_{1} \theta_{1}} B^{(n)}_{2} + o(1).
\end{align*}
Similarly, for $i=2,3,\ldots,K$, it follows from \eq{MLQ-BAR1}, \eq{MLQ-Dn-j-} and \lem{MLQ-En-j} that
\begin{align*}
 \frac 12 \sigma_{i}^{2} (\beta_{i} + \theta_{i}) \psi^{(n)}_{i}(\theta_{i}) & = - n \alpha^{(n)}_{e} \theta_{i}^{-1} E^{(n)}_{\Delta}(\theta_{i}\vcn{e}_{i}) + o(1)\nonumber\\
 & = - n^{1/2} (e^{\ell_{i-1} \theta_{i}} \Delta^{(n)}_{i-1} - e^{\ell_{i} \theta_{i}} \Delta^{(n)}_{i}) + o(1)\nonumber\\
 & = e^{\ell_{i-1} \theta_{i}} B^{(n)}_{i} - e^{\ell_{i} \theta_{i}} B^{(n)}_{i+1} + o(1).
\end{align*}
Thus, we have \eq{MLQ-BAR2-i}.
\end{proof}

For $i \in J_{K-1}$, $\psi^{(n)}_{i}(\theta_{i})$ is finite for any $\theta_{i} \in \dd{R}_{+}$, so we can substitute $\theta_{i} = - \beta_{i}$ in \eq{MLQ-BAR2-i}, and we have
\begin{align}
\label{eq:MLQ-Bi1}
  e^{\beta_{i} (\ell_{i} - \ell_{i-1})} B^{(n)}_{i} - B^{(n)}_{i+1} = o(1), \qquad i \in J_{K-1}.
\end{align}
Recall that $\xi_{j} = \prod_{i=1}^{j} e^{\beta_{i}(\ell_{i} - \ell_{i-1})}$, then \eq{MLQ-Bi1} implies that
\begin{align}
\label{eq:MLQ-Bi2}
 & \xi_{i} B^{(n)}_{i} = \xi_{i-1} B^{(n)}_{i+1} + o(1), \qquad i \in J_{K-1},\\
\label{eq:MLQ-bi1}
 & \xi_{i} b_{i} d^{(n)}_{i} = (\xi_{i-1} - \xi_{i}) B^{(n)}_{i+1} + o(1), \qquad i \in J_{K-1}.
\end{align}
Recall the definition \eq{ci} of $c_{i}$, then \eq{MLQ-bi1} can be written as
\begin{align}
\label{eq:dni-1}
 & d^{(n)}_{i} = c_{i} \xi_{i}^{-1} B^{(n)}_{i+1} + o(1), \qquad i \in J_{K-1},
\end{align}
where we note that $c_{i} < 0$ for $i \in J_{K-1}$.

From \eq{MLQ-Bi2}, we have $B^{(n)}_{i+1} = \xi_{i+1}^{-1} \xi_{i} B^{(n)}_{i+2} + o(1)$ for $i = 0,1,\ldots,K-2$, and, substituting this $B^{(n)}_{i+1}$ into \eq{dni-1}, we have, for each $i \in J_{K-1}$,
\begin{align*}
   & d^{(n)}_{i} =c_{i} \xi_{i+1}^{-1} B^{(n)}_{i+2} + o(1).
\end{align*}
Repeatedly applying this substitution and using the fact that $B^{(n)}_{K} = b_{K} d^{(n)}_{K}$, we arrive at
\begin{align}
\label{eq:dni-2}
   & d^{(n)}_{i} = c_{i} \xi_{K-1}^{-1} B^{(n)}_{K} + o(1) = c_{i} \xi_{K-1}^{-1} b_{K} d^{(n)}_{K} + o(1), \qquad i \in J_{K-1}.
\end{align}
Substituting this $d^{(n)}_{i}$ into $\sum_{i=1}^{K} d^{(n)}_{i} = 1$, we have
\begin{align*}
  d^{(n)}_{K} \left(\sum_{i=1}^{K-1} c_{i} \xi_{K-1}^{-1} b_{K} + 1 \right) = 1 + o(1).
\end{align*}
Recall the definition of $d_{K}$ by \eq{di}, then letting $n \to \infty$ in this equation, we have $ \lim_{n \to \infty} d^{(n)}_{K} = d_{K}$ . This and \eq{dni-2} yield
\begin{align}
\label{eq:dni-lim}
  \lim_{n \to \infty} d^{(n)}_{i} = d_{i}, \qquad i \in J_{K}.
\end{align}
where $d_{i}$ must be positive because $c_{i} b_{K} > 0$. Thus, we have proved that $\{\nu^{(n)}; n \ge 1\}$ is tight.

\noindent {(\bf 2nd Step)} $\;$ We prove that $\nu^{(n)}_{i} \Rightarrow \nu_{i}$ for $i \in J_{K}$ in \eq{h-nu-lim}. Since we already know \eq{dni-lim}, this weak convergence is equivalent to the following convergence of moment generating function $\varphi^{(n)}_{i}$ of $\nu^{(n)}_{i}$.
\begin{align*}
  \lim_{n \to \infty} \varphi^{(n)}_{i}(\theta_{i}) = \widetilde{h}_{i}(\theta_{i}), \qquad i \in J_{K},
\end{align*}
equivalently, by \cor{MLQ-gn-limit}, 
\begin{align}
\label{eq:psi/d-n-limit}
  \lim_{n \to \infty} \psi^{(n)}_{i}(\theta_{i})/d^{(n)}_{i} = \widetilde{h}_{i}(\theta_{i}), \qquad \vc{\theta} \in \dd{R}^{K-1} \times \dd{R}_{-}, \; i \in J_{K},
\end{align}
where $\widetilde{h}_{i}$ is the moment generating function of the probability density function $h_{i}$ of \eq{hj}, that is,
\begin{align}
\label{eq:thj}
 & \widetilde{h}_{i}(\theta_{i}) = \begin{cases}
 \ds u_{i}(\theta_{i}) 1(b_{i} = 0) + v_{i}(\theta) 1(b_{i} \not= 0), & x \in S_{i}, i \in J_{K-1}, \\
 \ds \frac {\beta_{K}} {\beta_{K} + \theta_{K}} e^{\theta_{K} \ell_{K-1}}, & x \in S_{K}, i = K, 
\end{cases}
\end{align}
where
\begin{align*}
  u_{i}(\theta_{i}) = \frac {e^{\theta_{i} \ell_{i}} - e^{\theta_{i} \ell_{i-1}}} {\theta_{i} (\ell_{i} - \ell_{i-1})}, \qquad v_{i}(\theta_{i}) = \frac {\beta_{i} (e^{\theta_{i} \ell_{i} + \beta_{i}(\ell_{i} - \ell_{i-1})} - e^{\theta_{i} \ell_{i-1}})} {(\beta_{i} + \theta_{i})(e^{\beta_{i} (\ell_{i} - \ell_{i-1})}- 1)}.
\end{align*}

We now compute $\psi^{(n)}(\theta_{i})$ for $i \in J_{K}$. For $i = K$, from \eq{MLQ-BAR2-i} of \lem{MLQ-BAR2-i},
\begin{align*}
  \psi^{(n)}_{K}(\theta_{K}) = \frac {\beta_{K}/b_{K} } {\beta_{K} + \theta_{K}} e^{\theta_{K} \ell_{K-1}} B^{(n)}_{K} = \frac {\beta_{K}} {\beta_{K} + \theta_{K}} e^{\theta_{K} \ell_{K-1}} d^{(n)}_{K}
\end{align*}
because $b_{K} < 0$ and $B^{(n)}_{K} = b_{K} d^{(n)}_{K}$. For $i \in J_{K-1}$, applying \eq{MLQ-Bi1} and \eq{dni-1} to \eq{MLQ-BAR2-i} of \lem{MLQ-BAR2-i}, we have
\begin{align}
\label{eq:psi-n-i}
  \frac 12 \sigma_{i}^{2}(\beta_{i} + \theta_{i}) \psi^{(n)}_{i}(\theta_{i}) & = \left(e^{\ell_{i-1} \theta_{i}} \xi_{i}^{-1} \xi_{i-1} - e^{\ell_{i} \theta_{i}} \right) B^{(n)}_{i+1} + o(1) \nonumber\\
  & = \left(e^{\ell_{i-1} \theta_{i}} \xi_{i-1} - e^{\ell_{i} \theta_{i}} \xi_{i} \right) \frac 1{c_{i}} d^{(n)}_{i}\nonumber\\
  & = \left(e^{\ell_{i-1} \theta_{i}} - e^{\ell_{i} \theta_{i} + \beta_{i} (\ell_{i} - \ell_{i-1})} \right) \frac 1{c_{i}} \xi_{i-1} d^{(n)}_{i} + o(1).
\end{align}
Hence, if $b_{i} \not= 0$, then, from \eq{psi-n-i} and $c_{i} = b_{i}^{-1} (1 - e^{\beta_{i} (\ell_{i} - \ell_{i-1})}) \xi_{i-1}$, we have
\begin{align*}
  \psi^{(n)}_{i}(\theta_{i}) & = \frac {2/\sigma_{i}^{2}} {\beta_{i} + \theta_{i}} \frac {e^{\theta_{i} \ell_{i-1}} - e^{\theta_{i} \ell_{i}} e^{\beta_{i} (\ell_{i} - \ell_{i-1})}} {1 - e^{\beta_{i}(\ell_{i} - \ell_{i-1})}} b_{i} d^{(n)}_{i} + o(1)\\
  & = \frac {\beta_{i}} {\beta_{i} + \theta_{i}} \frac {e^{\theta_{i} \ell_{i-1}} - e^{\theta_{i} \ell_{i}} e^{\beta_{i} (\ell_{i} - \ell_{i-1})}} {1 - e^{\beta_{i}(\ell_{i} - \ell_{i-1})}} d^{(n)}_{i} + o(1),
\end{align*}
and next assume that $b_{i} = 0$, then, from \eq{psi-n-i} and $c_{i} = 2\sigma_{i}^{-2} (\ell_{i-1} - \ell_{i}) \xi_{i-1}$, we have
\begin{align*}
   &\psi^{(n)}_{i}(\theta_{i}) = \frac {e^{\theta_{i} \ell_{i}} - e^{\theta_{i} \ell_{i-1}}} {\theta_{i}(\ell_{i} - \ell_{i-1})} d^{(n)}_{i} + o(1), \qquad i \in J_{K-1}.
\end{align*}
Hence, this and \eq{dni-lim} yield \eq{h-nu-lim}. Thus, the proof of \thr{MLQ-main} is completed.

\section{Concluding remarks}
\label{sec:concluding}
\setnewcounter

As discussed in \sectn{introduction}, \thr{MLQ-main} can be proved from the tightness of the sequence of the stationary distributions in heavy traffic and the lemma below.

\begin{lemma}
\label{lem:th1}
Let $X^{(n)}(\cdot)$ be the Markov process of the multi-level queue indexed by $n \ge 1$ whose initial state $X^{(n)}(0)$ is $(L^{(n)}(0),T_{A}(0),T_{S}(0))$. Assume the conditions (\sect{preliminaries}.a)--(\sect{preliminaries}.f), and define the diffusion scaled process $\widehat{X}^{(n)}(\cdot) \equiv \{\widehat{X}^{(n)}(t); t \ge 0\}$ by
\begin{align*}
  \widehat{X}^{(n)}(t) = n^{-1/2} X^{(n)}(nt), \qquad t \ge 0,
\end{align*}
and let $Z(\cdot)$ be the unique weak solution of \eq{SIE-Z}. (i) If the distributions of $\widehat{X}^{(n)}(0)$ for $n \ge 1$ is tight, then $\widehat{X}^{(n)}(\cdot) \Rightarrow \{(Z(t),0,0); t \ge 0\}$ as $n \to \infty$, where `$\Rightarrow$'' stands for the weak convergence of a process. (ii) In particular, if $\widehat{X}^{(n)}(\cdot)$ is a stationary process and if the distribution of $\widehat{L}^{(n)}(0)$ for $n \ge 1$ is tight, then the distribution of $\widehat{X}^{(n)}(0)$ for $n \ge 1$ is tight, and $L^{(n)}(\cdot) \Rightarrow Z(\cdot)$, where $Z(\cdot)$ is also a stationary process.
\end{lemma}
\begin{remark}
\label{rem:lem-th1}
Note that $\widehat{X}^{(n)}(0) = n^{-1/2} X^{(n)}(0)$, in particular, $\widehat{L}^{(n)}(0) = n^{-1/2} L^{(n)}(0)$, and therefore the stationary distributions of $\widehat{X}^{(n)}(t)$ and $n^{-1/2} X^{(n)}(t)$ are identical.
\end{remark}

(i) of this lemma is identical with Theorem 2.1 of \cite{AtarMiya2025} if $L^{(n)}(0) = 0$ and $(T_{A}(0), T_{S}(0))$ has the same distribution as $(T_{A}(j), T_{S}(j))$ for $j \ge 1$, but it is not hard to see (i) and (ii). For completeness, we outline the proof of \lem{th1} in \app{lem-th1}. In this proof, the BAR approach is used to derive \eq{Rne bound}, but it can be proved without BAR.

Thus, \thr{MLQ-main} is proved if the tightness holds for the sequence of the stationary queue length distributions in heavy traffic. In the proof of \thr{MLQ-main}, we already show this tightness by the BAR approach, but it may be questioned whether it can be obtained not using the BAR. A typical approach for this is based on the condition (5.1) in Proposition 5.1 of \citet{DaiMeyn1995}. To prove this condition, Lipschitz continuity of Skorohod map has a pivotal role (e.g., see \cite{BudhLee2009,GamaZeev2006}).

Unfortunately, we cannot use Skorohod map for the multi-level queue because its arrivals and service depend on queue length. Instead of this map, the semimartingale decomposition of \cite{DaleMiya2019} may be used (see \app{lem-th1} for its details). However, we could not prove the condition of \cite{DaiMeyn1995} by this method. We leave this line of the proof for future work, which will be useful when the BAR approach does not fully prove the tightness as in the case of the 2-level $GI/G/1$ queue.

\appendix

\section*{Appendix}
\setnewcounter
\setcounter{section}{1}

\subsection{Outline of the proof of \lem{th1}}
\label{app:lem-th1}
\setnewcounter

In this section, we use the notations in \lem{th1}, and outline its proof. For each $n \ge 1$, \eq{Ln-t} can be written as
\begin{align}
\label{eq:hLn-t}
  \widehat{L}^{(n)}(t) = \widehat{L}^{(n)}(0) + \widehat{N}^{(n)}_{e}(t) - \widehat{N}^{(n)}_{d}(t), \qquad t \ge 0,
\end{align}
where $\widehat{N}^{(n)}_{e}(t) = n^{-1/2} N^{(n)}_{e}(nt)$ and $\widehat{N}^{(n)}_{d}(t) = n^{-1/2} N^{(n)}_{e}(nt)$.

The key idea of \cite{AtarMiya2025} is to decompose the counting processes $\widehat{N}^{(n)}_{e}(\cdot)$ and $\widehat{N}^{(n)}_{d}(\cdot)$ as non-standard semimartingales obtained in \cite{DaleMiya2019}. Following \cite{AtarMiya2025}, we first present this decomposition for the delayed renewal processes $A$ and $S$ in which $T_{A}(0)$ and $T_{S}(0)$ are initial delays. For this, we need the following notations.
\begin{align*}
 & R_{A}(t) = \inf\left\{ u > 0; A(t+u) > A(t)\right\}, \qquad R_{S}(t) = \inf\left\{ u > 0; S(t+u) > S(t)\right\},\\
 & M_{A}(t) = \sum_{j=1}^{A(t)} (1 - T_{A}(j)), \qquad M_{S}(t) = \sum_{j=1}^{A(t)} (1 - T_{S}(j)).
\end{align*}
Note that $R_{A}(t)$ and $R_{S}(t)$ are the residual times to the next counting instants of renewal processes $A$ and $S$, respectively, at time $t$, and $M_{A}(\cdot)$ and $M_{S}(\cdot)$ are martingales with respect to the filtration generated by $\{(A(t),R_{A}(t),S(t),R_{S}(t)); t \ge 0\}$. Then, by Theorem 2.1 and Remark 3.3 of \cite{DaleMiya2019}, we have the semimartingales:
\begin{align}
\label{eq:DM-decomposition}
  A(t) = t + R_{A}(t) - T_{A}(0) + M_{A}(t), \quad S(t) = t + R_{S}(t) - T_{S}(0) + M_{S}(t),
\end{align}
which are called Daley-Miyazawa decompositions in \cite{AtarMiya2025}. Note that $T_{A}(0)$ and $T_{S}(0)$ are assumed to have the same distributions as $T_{A}$ and $T_{S}$, respectively, in Theorem 2.1 of \cite{AtarMiya2025}. However, in the settings of \lem{th1}, $T_{A}(0)$ and $T_{S}(0)$ may depend on the system index $n$, so we will need a condition for them to be negligible under diffusion scaling, that is, $n^{-1/2} T_{A}(0)$ and $n^{-1/2} T_{S}(0)$ vanishes $a.s.$ as $n \to \infty$. For this condition to hold, we will use the following condition.
\begin{mylist}{-3}
\item [(A1)] There are nonnegative random variable $\widetilde{T}_{A}$ and $\widetilde{T}_{S}$ such that they are stochastically bounded by $T_{A}(0)$ and $T_{S}(0)$, respectively, and have distributions which do not depend on the system index $n$ and are independent of $\{T_{A}(j); j \ge 1\}$ and $\{T_{S}(j); j \ge 1\}$, where, for random variables $X$ and $Y$, $X$ is said to be stochastically bounded by $Y$ if $\dd{P}[X>x] \le \dd{P}[Y>x]$ for any $x \in \dd{R}$.
\end{mylist}

Thus, we can see that the proofs of \cite{AtarMiya2025} work under (A1). Namely, we have the following lemma, whose proof is outlined for completeness.
\begin{lemma}
\label{lem:th21}
Theorem 2.1 of \cite{AtarMiya2025} is still valid if $\widehat{L}^{(n)}(0) = 0$ for all $n \ge 1$ and if the initial delays $T_{A}(0)$ and $T_{S}(0)$ satisfy the condition (A1).
\end{lemma}

\begin{proof}[Outline of the proof of \lem{th21}]
We decompose $\widehat{N}^{(n)}_{e}(\cdot)$ and $\widehat{N}^{(n)}_{d}(\cdot)$ similarly to \eq{DM-decomposition}. To this end, define
\begin{align*}
 & U^{(n)}(t) = \sum_{i = 1}^{K} \int_{0}^{t} \lambda^{(n)}_{i} 1(L^{(n)}(u) \in S^{(n)}_{i}) du\\
 & V^{(n)}(t) = \sum_{i = 1}^{K} \int_{0}^{t} \mu^{(n)}_{i} 1(L^{(n)}(u) \in S^{(n)}_{i}) du - \int_{0}^{t} \mu^{(n)}_{1} 1(L^{(n)}(u) = 0) du,
\end{align*}
then $N^{(n)}_{e}(t)$ and $N^{(n)}_{d}(t)$ can be written as
\begin{align*}
  \widehat{N}^{(n)}_{e}(t) = n^{-1/2} A\left(U^{(n)}(nt)\right), \qquad \widehat{N}^{(n)}_{d}(t) = n^{-1/2} S\left(V^{(n)}(nt)\right).
\end{align*}
Hence, it follows from \eq{DM-decomposition}, we have
\begin{align*}
  \widehat{N}^{(n)}_{e}(t) = n^{-1/2} \left(U^{(n)}(nt) + R_{A}(U^{(n)}(nt)) - T_{A}(0) + M_{A}(U^{(n)}(nt))\right),\\
  \widehat{N}^{(n)}_{d}(t) = n^{-1/2} \left(V^{(n)}(nt) + R_{S}(V^{(n)}(nt)) - T_{S}(0) + M_{S}(V^{(n)}(nt))\right).
\end{align*}
Substituting these expressions into the diffusion scaled version of \eq{Ln-t}, we have
\begin{align}
\label{eq:tLn-t}
  \widehat{L}^{(n)}(t) & = \widehat{L}^{(n)}(0)  - n^{-1/2} (T_{A}(0) - T_{S}(0)) + \sum_{i = 1}^{K} \int_{0}^{t} (b_{i} + o(1)) 1(L^{(n)}(nu) \in S^{(n)}_{i}) du \nonumber\\
  & \quad + n^{-1/2} (R_{A}(U^{(n)}(nt)) - R_{S}(U^{(n)}(nt))) + \widehat{M}^{(n)}(t),
\end{align}
where
\begin{align*}
  \widehat{M}^{(n)}(t) = n^{-1/2} \left( M_{A}(U(nt)) + M_{S}(V(nt)) \right),
\end{align*}
which is shown to be a martingale with respect to an appropriately chosen filtration by Lemma 3.1 of \cite{AtarMiya2025}. Since $n^{-1/2} (T_{A}(0) - T_{S}(0))$ is negligible as $n \to \infty$ by the condition (A1), assuming $\widehat{L}^{(n)}(0) = 0$, it is shown in Section 4 of \cite{AtarMiya2025} that $\widehat{X}^{(n)}(\cdot)$ weakly converges to $\{(Z(t),0,0); t \ge 0\}$ as $n \to \infty$.
\end{proof}

\begin{proof}[Outline of the proof of \lem{th1}]
Assume that $\widehat{X}^{(n)}(\cdot)$ is a stationary process, and we aim to apply \lem{th21}. In this case, the initial delays $T_{A}(0)$ and $T_{S}(0)$ depend on the system index $n$, so we need to verify the condition (A1). For this, note that
\begin{align*}
  R^{(n)}_{e}(t) = R_{A}(U^{(n)}(t)), \qquad R^{(n)}_{d}(t) = R_{S}(V^{(n)}(t)), \qquad t \ge 0,
\end{align*}
and $R^{(n)}_{e}(\cdot)$ and $R^{(n)}_{d}(\cdot)$ are stationary processes. This implies that 
\begin{align*}
  R^{(n)}_{e}(0) = R_{A}(U^{(n)}(0)) = T_{A}(0).
\end{align*}
Hence, we need to choose the distribution of $T_{A}(0)$ to be the stationary distribution of $R^{(n)}_{e}(\cdot)$, which indeed depends on the index $n$. However, by \lem{MLQ-moment}, 
\begin{align}
\label{eq:Rne bound}
  \dd{P}[R^{(n)}_{e}(0) > x] = \dd{P}[R^{(n)}_{e} > x] \le C_{e} \dd{E}[(T_{A}-x) 1(T_{A} > x)], \qquad x \ge 0,
\end{align}
where $C_{e} = \sup_{n \ge 1} (\min_{i \in J_{K}} \lambda^{(n)}_{i})^{-1} \alpha^{(n)}_{e}$, which is finite by (\sect{preliminaries}.c) and \cor{alpha 2}. Note that the right-hand side of \eq{Rne bound} is non-increasing in $x \ge 0$. Hence, let $\widetilde{T}_{A}$ be a random variable subject to the distribution determined by
\begin{align}
\label{eq:tTA-x}
  \dd{P}[\widetilde{T}_{A} > x] = 1 \wedge (C_{e} \dd{E}[(T_{A}-x) 1(T_{A} > x)]), \qquad x \ge 0.
\end{align}
Then, $T_{A}(0) (= R^{(n)}_{e}(0))$ is stochastically bounded uniformly in $n \ge 1$ by the random variable $\widetilde{T}_{A}$. Similarly, $\widetilde{T}_{S}$ can be found for $T_{S}(0)$.

Thus, the condition (A1) is satisfied, and it follows from \lem{th21} that $\{\widehat{L}^{(n)}(t) - \widehat{L}^{(n)}(0)  - n^{-1/2} (T_{A}(0) - T_{S}(0)); t \ge 0 \}$ weakly converges to $Z(\cdot)$ of \eq{SIE-Z} with $Z(0)=0$. Hence, it can be proved that, if the distributions of $\widehat{L}^{(n)}(0)$ for $n \ge 1$ is tight, then $\widehat{X}^{(n)}(\cdot) \Rightarrow \{(Z(t),0,0); t \ge 0\}$. It remains to prove that $Z(\cdot)$ is a stationary process. This easily follows from the fact that
Thus, the proof of \lem{th1} is completed.
\end{proof}

\begin{remark}
\label{rem:th21}
$\widetilde{T}_{A}$ and $\widetilde{T}_{S}$ of \lem{th21} do not need to have finite means, but we can see these finiteness. Namely, let $x_{0} = \inf\{x \ge 0; C_{e} \dd{E}[(T_{A}-x) 1(T_{A} > x)] \ge 1\}$, then it follows from \eq{tTA-x} that
\begin{align*}
  \dd{E}[\widetilde{T}_{A}] \le x_{0} + C_{e} \int_{0}^{\infty} \dd{E}[(T_{A}-x) 1(T_{A} > x)] dx = x_{0} + \frac {C_{e}} 2 \dd{E}[T_{A}^{2}] < \infty,
\end{align*}
because $\sigma_{A}^{2} < \infty$. Similarly, $\dd{E}[\widetilde{T}_{A}] < \infty$.
\end{remark}

\subsection*{Acknowledgement} We are grateful to Rami Atar for discussions about \lem{th1}. Yutaka Sakuma is supported by JSPS KAKENHI Grant Number JP22K11927.


\def\cprime{$'$} \def\cprime{$'$} \def\cprime{$'$} \def\cprime{$'$}
  \def\cprime{$'$} \def\cprime{$'$} \def\cprime{$'$}

\end{document}